\documentclass[12pt]{article}
\usepackage[a4paper,margin=25mm]{geometry}
\usepackage{amsmath}
\usepackage{amssymb}
\usepackage{amsthm}
\usepackage{framed}
\usepackage{xcolor}

\theoremstyle{plain}
\newtheorem{thm}{Theorem}[section]
\newtheorem{cor}[thm]{Corollary}

\newtheorem{lemma}[thm]{Lemma}
\theoremstyle{definition}
\newtheorem*{adef}{Definition}
\theoremstyle{remark}

\newcommand{\s}{z}
\DeclareMathOperator{\lc}{lc}
\DeclareMathOperator{\rlc}{rlc}
\DeclareMathOperator{\ac}{lc}
\DeclareMathOperator{\rac}{rlc}
\DeclareMathOperator{\val}{val}
\newcommand{\LL}{\mathcal{L}}
\newcommand{\eps}{\varepsilon}
\renewcommand{\P}{\ensuremath\mathsf{P}}
\newcommand{\NP}{\ensuremath\mathsf{NP}}
\newcommand{\APX}{\ensuremath\mathsf{APX}}
\newcommand{\size}[1]{\textrm{size}(#1)}
\newcommand{\ints}{\mathbb{Z}}
\newcommand{\nats}{\mathbb{N}}
\newcommand{\rats}{\mathbb{Q}}

\DeclareMathOperator{\opt}{opt}

\title{The complexity of solution-free sets of integers}
\author{
Keith J. Edwards   \\
Computing   \\
University of Dundee   \\
Dundee, DD1 4HN   \\
United Kingdom \\
\texttt{kjedwards@dundee.ac.uk}  \\
\ \\
Steven D. Noble \\
Department of Economics, Mathematics and Statistics \\
Birkbeck, University of London \\
Malet Street \\
London, WC1E 7HX \\
United Kingdom \\
\texttt{s.noble@bbk.ac.uk}  \\
}

\begin{document}

\maketitle

\begin{abstract}
Given a linear equation $\LL$, a set $A $ of integers is $\LL$-free if
$A$ does not contain any non-trivial solutions to $\LL$. Meeks and
Treglown~\cite{meeks-treglown} showed that for certain kinds of linear
equations, it is $\NP$-complete to decide if a given set of integers
contains a solution-free subset of a given size. Also, for equations
involving three variables, they showed that the problem of determining
the size of the largest solution-free subset is $\APX$-hard, and that for
two such equations (representing sum-free and progression-free sets),
the problem of deciding if there is a solution-free subset with at least
a specified proportion of the elements is also $\NP$-complete.

We answer a number of questions posed by Meeks and Treglown,
by extending the results above to all linear equations, and showing
that the problems remain hard for sets of integers whose elements are
polynomially bounded in the size of the set. For most of these
results, the integers can all be positive as long as the coefficients
do not all have the same sign.

We also consider the problem of counting the number of solution-free
subsets of a given set, and show that this problem is $\#P$-complete
for any linear equation in at least three variables.
\end{abstract}
\newpage
\section{Introduction}
There has been much study of sum-free or progression-free sets of
integers, that is, sets of integers containing no solution to the
equations $x + y = z$ or $x + z = 2y$ respectively, or, more
generally, sets containing no solution to some other linear equation.

In a recent paper, Meeks and Treglown~\cite{meeks-treglown} initiated
the study of the computational complexity aspects of these problems,
by considering several computational problems concerning solution-free
sets, or, more precisely, solution-free subsets of a given set of
integers.

\subsection{Solution-free sets} \label{sec:background}
We start with a brief account of the notation and terminology that
will be used throughout the paper.

Consider a fixed linear equation
$\LL$ of the form
\begin{align}\label{general-linear-eqn}
c_1x_1+\dots +c_\ell x_\ell =K,
\end{align}
where $c_1,\dots, c_\ell,K \in \mathbb{Z}$.
The equation $\LL$ is \emph{homogeneous} if $K=0$ and inhomogeneous
otherwise. It is \emph{translation-invariant} if
\[
\sum_{i=1}^{\ell} c_i=0 .
\]
We will be interested in solutions to such an equation in some set of
integers. In general, a solution to $\LL$ in a set
$A \subseteq \mathbb{Z}$
is a sequence $(x_1,\ldots, x_{\ell}) \in A^{\ell}$
which satisfies the equation.

However, some equations have trivial solutions which must be excluded
from consideration in order to obtain sensible questions.
Suppose that $\LL$ is homogeneous and translation-invariant, that is, it is of the form
\begin{align}\label{hom-trans-inv-eqn}
c_1x_1+\ldots +c_\ell x_\ell = 0,
\end{align}
where $c_1+\ldots+c_\ell = 0$.
Then $(x,\dots, x)$ is
a trivial solution of~\eqref{hom-trans-inv-eqn} for any $x$ (and so no
non-empty set of integers can be solution-free for $\LL$).
More generally, a solution $(x_1, \dots , x_{\ell})$ to $\LL$ is said to be
\emph{trivial} if there exists a partition $P_1, \dots ,P_k$ of
$\{1, \ldots, \ell\}$ so that:
\begin{itemize}
\item[(i)] $x_i=x_j$ for every $i,j$ in the same partition class $P_r$;
\item[(ii)] For each $r \in \{1, \ldots ,k \}$, $\sum_{i \in P_r} c_i=0$.
\end{itemize}
For any linear equation $\LL$, a set $A $ of integers is
\emph{$\LL$-free} if $A$ does not contain
any non-trivial solutions to $\LL$.

\section{Computational Problems} \label{sec:comp-problems}
Meeks and Treglown~\cite{meeks-treglown} considered several
computational problems concerning $\LL$-free sets for some fixed linear
equation $\LL$ (they restricted attention to homogeneous equations).

The main problem considered was the following:
\begin{framed}
\noindent $\LL$-\textsc{Free Subset}\newline
\textit{Input:} A finite set $A \subseteq \mathbb{Z}$ and $k \in
\mathbb{N}$.\newline
\textit{Question:} Does there exist an $\LL$-free subset
$A' \subseteq A$ such that $|A'|=k$?
\end{framed}

They show that the problem is $\NP$-complete in some cases:
\begin{thm}[Meeks \& Treglown \cite{meeks-treglown}]\label{npthm}
Let $\LL$ be a linear equation of the form $a_1x_1+\dots+a_\ell x_\ell
= b y$ where each $a_i \in \mathbb{N}$ and $b \in \mathbb{N}$ are fixed
and $\ell\geq 2$.
Then $\LL$-\textsc{Free Subset} is $\NP$-complete.
\end{thm}

They also consider the problem of finding the size of the largest
$\LL$-free subset of a set:
\begin{framed}
\noindent \textsc{Maximum $\LL$-Free Subset}\newline
\textit{Input:} A finite set $A \subseteq \mathbb{Z}$.\newline
\textit{Question:} What is the cardinality of the largest $\LL$-free subset $A' \subseteq A$?
\end{framed}

This problem is shown to be $\APX$-hard in some limited cases with three
variables:
\begin{thm}[Meeks \& Treglown \cite{meeks-treglown}]
Let $\LL$ be a linear equation of the form $a_1x_1 + a_2x_2 = by$,
where $a_1,a_2,b \in \mathbb{N}$ are fixed. Then
\textsc{Maximum $\LL$-Free Subset} is $\APX$-hard.
\end{thm}

Meeks and Treglown also consider the problem of determining if a set
has an $\LL$-free subset containing some fixed proportion of the
elements:
\begin{framed}
\noindent $\eps$-$\LL$-\textsc{Free Subset}\newline
%
%
%
%
\textit{Input:} A finite set $A \subseteq \mathbb{Z} \setminus \{0\}$.\newline
\textit{Question:} Does there exist an $\LL$-free subset $A' \subseteq
A$ such that $|A'|\geq \eps|A|$?
\end{framed}
This problem is shown to be $\NP$-complete for some values of $\eps$
for two kinds of equations, namely $x + y = z$ and $ax+by=cz$, where
$a,b,c > 0$ and $a + b = c$.

The authors raise a number of questions, including the following:
\begin{enumerate}
\item Can these results be proved for the case when the set $A$ has elements of
size polynomial in the size of $A$ (rather than exponential as is the
case with the results proved in~\cite{meeks-treglown})?
\item Can the results be generalised to other linear equations, in
particular equations such as $x+y=z+w$?
\end{enumerate}
We answer both these questions in the affirmative, by extending all
the above results to all linear equations (both homogeneous and
inhomogeneous) with at least
three variables, using only polynomially-sized integers.

We close by proving the hardness of the counting version of $\LL$-free subset.
\section{Representing a graph as a set}
In this section we will describe constructions which allow us to
represent a graph as a set of integers relative to a linear equation.
The constructions
are inspired by the hypergraph one used in~\cite{meeks-treglown},
but are different in a number of respects; in particular, we do not
use hypergraphs. For a graph with $\ell$ variables, instead of
associating $\ell-1$ variables with vertices and the remaining
(dependent) one
with an edge of an $\ell-1$-uniform hypergraph, we associate two
variables with vertices and the remaining $\ell-2$ with an edge, one
of which is ``dependent'' and the rest ``free''.
We will start with homogeneous equations.

We first show how to rearrange a homogeneous equation into a more convenient
form. Note that this only involves re-ordering the terms.
\begin{lemma} \label{make-standard}
Let $\LL$ be a linear equation
$c_1 x_1 + \cdots + c_{\ell} x_{\ell} = 0$, where $\ell \ge 3$
and the coefficients $c_i$ are all non-zero.
Then we can find an equivalent
equation $\LL'$:
$a_1 x_1 + a_2 x_2 + b_1 y_1 + \cdots + b_{\ell-3} y_{\ell-3} = b_0y_0$,
with the following properties:
\begin{enumerate}
\item $a_1 \le a_2 \le b_1 \le \ldots \le b_{\ell-3}$.
\item $b_0 > 0$ unless all of $c_1, \ldots , c_{\ell}$ have the same
sign, in which case $b_0 < 0$.
\item All of $a_1, a_2, b_1, \ldots , b_{\ell-3}$ are non-zero,
\item If $C = \{ a_1, a_2, b_1, \ldots , b_{\ell-3}\}$, then
$\sum_{x \in C, x < 0} (-x) < \sum_{x \in C, x > 0} x$, that is, the
negative coefficients on the left-hand side have smaller total size
than the positive coefficents on the left-hand side.
\item The last two coefficients on the left-hand side are positive
(these are $a_1,a_2$ if $\ell=3$, $a_2,b_1$ if $\ell=4$,
$b_{\ell-4},b_{\ell-3}$ if $\ell>4$).
\end{enumerate}
\end{lemma}

\begin{proof}
First, if all the coefficients have the same sign, then we can assume
they are all strictly positive. Order the coefficients in increasing
order, and set $a_1 = c_1$, $a_2 = c_2$,
$b_i = c_{i+2}$ for $i = 1, \ldots , \ell-3$, and $b_0 = -c_{\ell}$.

Now suppose that not all the coefficients have the same sign.
If there is only one positive coefficient, or only one negative,
we can reverse signs if necessary to ensure there is exactly one
negative coefficient. Otherwise, by
reversing the signs of all the coefficients if necessary, we can
assume that (the negation of) the sum of the negative coefficients is at
most the sum of the positive coefficients. In either case,
reorder the terms so that
the coefficients are in non-decreasing order, then the first term has
a negative coefficient since not all coefficients have the same sign;
move this term to the right-hand side.
Then rename the coefficients and variables to obtain the required
form.
\end{proof}

Note: We will say that the equation $\LL'$ is in standard
form.

We now prove a lemma which shows that a graph may be represented
as a set of integers
relative to a fixed homogeneous linear equation. In this and later
variants, each vertex and each edge is represented by one or more
integers; these integers are all distinct.
\begin{lemma} \label{construct-set}
Consider a linear equation
$c_1 x_1 + \cdots + c_{\ell} x_{\ell} = 0$, where $\ell \ge 3$
and the coefficients $c_i$ are all non-zero, and let
$\LL$ be an equivalent linear equation
$a_1 x_1 + a_2 x_2 + b_1 y_1 + \cdots + b_{\ell-3} y_{\ell-3} = b_0y_0$
in standard form.
Let $G = (V,E)$ be a graph, and let $n = |V|$.
Then we can construct in polynomial time a
set $A \subseteq \mathbb{Z}$ with the following properties:
\begin{enumerate}
\item $A$ is the union of $\ell-1$ sets
$A_V, A_E^0 , A_E^1 , \ldots , A_E^{\ell-3}$, where
$A_V = \{x_v : v \in V\}$ and
$A_E^i = \{y_e^i : e \in E\}$ for each $i = 0, 1, \ldots , \ell-3$;
\item $A = |V| + (\ell-2)|E|$; \label{distinct}
\item for every edge $e = (v,w)$, some permutation of
$(x_v,x_w, y_e^1 , \ldots , y_e^{\ell-3} , y_e^0)$ is a solution to
$\LL$; such a solution we call an edge-solution; \label{edge-sol}
\item if $(z_1, \ldots , z_{\ell})$ is a non-trivial solution to
$\LL$, then $(z_1, \ldots , z_{\ell})$ is a permutation of
$(x_v,x_w, y_e^1 , \ldots , y_e^{\ell-3} , y_e^0)$ for some edge $e = (v,w)$;
\label{no-other-sols}
\item $\max(A) = \mathcal{O}(|V|^{2\ell+2})$; \label{poly-sized}
\item $A \subseteq \mathbb{N}$ unless all the coefficients
$c_1 , \ldots , c_{\ell}$ have the same sign.
\end{enumerate}
\end{lemma}

\begin{proof}
We need to choose the elements of $A$ to satisfy the conditions~1--6.
The set $A$ consists of a set $A_V$ of vertex labels (one per
vertex) and a set $A_E = A_E^0 \cup A_E^1 \cup \ldots \cup
A_E^{\ell-3}$ of edges labels ($\ell-2$ per edge).

In order to choose the labels we will use a set of variables, one for each
element of $A$, and construct a set of constraints (inequalities and
equations) which these must satisfy. We then choose values for each variable so
that all the constraints are satisfied. We will use upper case for
each variable name and the corresponding lower case for its value.
%
%
The set of variables is
$B = B_V \cup B_E$, where
$B_E = B_E^0 \cup B_E^1 \cup \cdots \cup B_E^{\ell-3}$,
$B_V = \{X_v : v \in V\}$ and
$B_E^i = \{ Y_e^i : e \in E\}$ for each $i = 0 , 1 , \ldots , \ell-3$.
The variable $X_v$ is associated with the vertex $v$, and the
variables $Y_e^0, \ldots , Y_e^{\ell-3}$ with the edge~$e$.
We will call the elements of $B_V$ vertex variables, those of
$B_E^1 \cup \cdots \cup B_E^{\ell-3}$ free edge variables, and those
of $B_E^0$ dependent edge variables.

We now consider the constraints that the variables must satisfy. Fix
an arbitrary ordering of the vertex set $V$.

First, for each edge $e = (v,w)$, we have a set of variables
corresponding to the edge and its endpoints, that is
$\{X_v, X_w, Y_e^1, \ldots , Y_e^{\ell-3}, Y_e^0\}$. We will call such
a set an edge-set of variables. To ensure that Condition~\ref{edge-sol} is satisfied,
we will require these to satisfy
the equation
\begin{equation} \label{con-equation}
a_1 X_v + a_2 X_w + b_1 Y_e^1 + \cdots + b_{\ell-3} Y_e^{\ell-3} =
b_0 Y_e^0,\tag{$*$}
\end{equation}
where $v < w$ in the vertex ordering.

Second, we have a set of inequalities of two types. In each case we
require that some linear combination of variables is non-zero.
Type~1 ensures
that all the labels are distinct, so that Condition~\ref{distinct} is satisfied,
while
Type~2 ensures that there are no non-trivial solutions to the
equation $\LL$, other than edge-solutions, so that Condition~\ref{no-other-sols} is satisfied.

For any sequence
$\s = (W_1, W_2, Z_1, \ldots, Z_{\ell-3}, Z_0) \in B^{\ell}$ of
variables, we say that $\s$ is an edge-sequence if
$\{W_1, W_2, Z_1, \ldots, Z_{\ell-3}, Z_0\}$ is an edge-set.

The inequalities are as follows:

Type~1: For any distinct pair $\s = (Z,Z') \in B^2$, we have the linear
combination $\lc(\s) = Z - Z'$, and require that $\lc(\s) \ne 0$.

Type~2: For any sequence
$\s = (W_1, W_2, Z_1, \ldots, Z_{\ell-3}, Z_0) \in B^{\ell}$, we have
the linear combination
$\lc(\s) = a_1 W_1 + a_2 W_2 + b_1 Z_1 + \cdots +
b_{\ell-3} Z_{\ell-3} - b_0 Z_0$. Provided that (i) $\s$ is not an
edge-sequence and (ii) $\lc(\s)$ is not identically zero, then we
require that $\lc(\s) \ne 0$. Note that $\lc(\s)$ being identically zero
corresponds to a trivial solution of the equation, where the total
coefficient of each variable is zero and so the terms cancel.

Now when choosing values to satisfy all of the inequalities, we must
take account of the constraints~\eqref{con-equation}.
Thus in each of the linear combinations $\lc(\s)$,
we substitute for
each dependent edge variable, that is, for any $e = (v,w)$, set
$Y_e^0 =
(a_1 X_{v} + a_2 X_w + b_1 Y_e^1 + \cdots + b_{\ell-3}
Y_e^{\ell-3})/b_0$ and collect terms to obtain a reduced
(rational) linear combination $\rlc(\s)$
of the vertex variables and free edge variables in the set
$B_V \cup B_E^1 \cup \cdots \cup B_E^{\ell-3}$.

%

It is clear that if values are chosen for all the variables,
satisfying the constraints~\eqref{con-equation}, then the values of
$\lc(\s)$ and $\rlc(\s)$ will be equal. Hence if each inequality
$\lc(\s) \ne 0$ is replaced by the
corresponding inequality $\rlc(\s) \ne 0$ and the values of the
variables are chosen so that these new inequalities are satisfied,
then the original inequalities, expressed in terms of $\lc(\s)$, must
also be satisfied.
Each inequality $\rlc(\s)\ne 0$ can be satisfied providing $\rlc(\s)$
is not identically zero.
Note however that potentially the reduced combination $\rlc(\s)$ is
identically zero even though $\lc(\s)$ is not, because terms may cancel
after substituting for the dependent edge variables $Y_e^0$. Hence
before we choose values for the variables, we must check that this
does not occur, that is,
we must verify that if $\s$ is not an edge-sequence and $\lc(\s)$
is not identically zero, then $\rlc(\s)$ is not identically zero.

%
%
%
%
For Type~1, $\rlc(\s)$ cannot be identically zero since the variables $Z,Z'$ are
distinct, and even if at least one is of the form $Y_e^0$, they cannot
cancel out.

For Type~2, consider any
sequence $\s = (W_1, W_2, Z_1, \ldots, Z_{\ell-3}, Z_0) \in
B^{\ell}$.
Suppose that $\rlc(\s)$
is identically zero. Then we need to show that either $\s$ is an
edge-sequence or $\lc(\s)$ is identically zero.

First suppose that $\ell \ge 4$ (so that each edge has at least one
free label). If none of
$W_1, W_2, Z_1, \ldots, Z_{\ell-3}, Z_0$ is a dependent edge variable,
then there is no substitution, so $\lc(\s) = \rlc(\s)$, hence $\lc(\s)$
is identically zero.

So suppose that at least one
of $W_1, W_2, Z_1, \ldots, Z_{\ell-3}, Z_0$ is a dependent edge
variable.
If the total coefficient of each dependent edge variable in $\lc(\s)$
is zero,
then the other terms are the same in $\rlc(\s)$ as in $\lc(\s)$, so must
reduce to zero, hence $\lc(\s)$ is identically zero.

Thus we may assume that $\s$ contains some dependent edge variable
$Y_e^0$ whose terms do not cancel, that is, the total coefficient
of $Y_e^0$ in $\lc(\s)$ is non-zero.
Then all of $Y_e^1, \ldots , Y_e^{\ell-3}$ must also occur
in the sequence $\s$, since if $Y_e^k$ does not occur in $\s$, the
coefficient of $Y_e^k$ in $\rlc(\s)$ is non-zero, contrary to
assumption.

Hence at least $\ell-2$ of the
terms in the sequence $\s$ are edge variables of $e$, leaving at most two other
terms. If one of them is
a (free or dependent) edge variable of some edge $e' \ne e$, then
$\rlc(\s)$ will have a non-zero coefficient of $Y_{e'}^k$ for
some $k \ge 1$ unless the final term is also an edge variable of $e'$.
But then there is an endpoint $v$ of $e$ which
is not an endpoint of $e'$, and the variable $X_v$ has non-zero
coefficient in $\rlc(\s)$, contrary to assumption.

Thus the two remaining terms in $\s$ must be vertex variables and
so must be the variables of the endpoints of $e$. But then $\s$ is an
edge-sequence.

Finally
consider the case when $\ell=3$. In this case $\lc(\s)$ is just
$a_1 W_1 + a_2 W_2 - b_0 Z_0$, with $a_1, a_2$ positive,
and for each edge $e = (v,w)$, the only edge variable is the
dependent one
$Y_e^0 = (a_1 X_v + a_2 X_w)/b_0$.
If exactly one of $W_1,W_2,Z_0$ is an
edge variable, then the other two must be the variables of its endpoints,
so $\s$ is an edge-sequence. If exactly
two are edge variables, these two terms cannot cancel in $\lc(\s)$,
for then the third term would be non-zero in $\rlc(\s)$.
So either the two terms come from the same edge and their sum has
non-zero coefficient of two vertex variables in $\rlc(\s)$, or they come
from distinct edges with at least three endpoints between them, and
then the sum of these two terms must have non-zero coefficient of at
least two vertex variables. 
In either case, since the remaining term is a variable of a single
vertex, at least one vertex has non-zero coefficient in $\rlc(\s)$.
Finally, suppose all three are edge
variables. If all three variables are from the same edge $e$, then
$\lc(\s) = c Y_e^0$ for some $c$, and since $\rlc(\s)$ is identically
zero, $c$ must be zero, so $\lc(\s)$ is also identically zero.
If one variable is from edge $e$, and two from edge $e'$, where
$e' \ne e$, then some vertex $v$ is an endpoint of $e$ but not of $e'$,
and $X_v$ has non-zero coefficient in $\rlc(\s)$, contrary to
assumption.
So all three variables are from distinct edges.
Then either there is a vertex which is an endpoint of just one of
them, so that its variable has non-zero coefficient in $\rlc(\s)$, or the
edges form a
triangle, in which case one of the vertices occurs only in $W_1$ and
$W_2$ and so its variable has non-zero coefficient in $\rlc(\s)$.


We conclude that if $\rlc(\s)$ is identically zero, then either
$\s$ is an edge-sequence or $\lc(\s)$ is identically zero.
We now estimate the number of inequalities which must be satisfied.
The number $L$ of vertex and edge labels is $n + (\ell-2)|E|$, so
$L = \mathcal{O}(n^2)$. Then the number $M$ of inequalities is at most
$\binom{L}{2} + L^{\ell}$ which is
$\mathcal{O}(n^{2\ell})$.

Order the edges arbitrarily.
We choose the values $\val(Z)$ of each variable $Z$, starting with
the vertex labels $x_v = \val(X_v)$ in order, then
for each edge $e$ in turn, choosing the
$\ell-3$ free labels $y_e^i = \val(Y_e^i)$ for $i = 1, \ldots ,
{\ell-3}$
(the dependent label $y_e^0$ will then be determined by
$y_e^0 = (a_1 x_{v} + a_2 x_w + b_1 y_e^1 + \cdots + b_{\ell-3}
y_e^{\ell-3}
)/b_0$).
We shall choose the value of each label to be an integer multiple of $b_0$ and
we must avoid any value
which would make one of the reduced linear combinations
equal to zero. When choosing each label, there is some subset of
the reduced linear combinations which contain the new label and whose
other labels have already been chosen. Each reduced linear
combination in this subset forbids one value for the new label (though
the forbidden value will be irrelevant if it is not an integer
multiple of $b_0$). Thus each label may be chosen from any set
containing at least $M+1$ positive integer multiples of $b_0$. We
choose the labels so that they are increasing; thus each successive
label is chosen from the next available block of $M+1$ integer
multiples of $b_0$. In
particular, none of the labels will be greater than $L(M+1)b_0$, which is
$\mathcal{O}(n^{2\ell+2})$, thus satisfying Condition~\ref{poly-sized}.

Finally, consider any dependent label $y_e^0$, where $e = (v,w)$ with $v<w$.
We need to check that provided not all of the coefficients of $\LL$
have the same sign, $y_e^0$ is a positive integer. Note that in this
case, $b_0 > 0$.
By definition,
$y_e^0 = (a_1 x_{v} + a_2 x_w + b_1 y_e^1 + \cdots + b_{\ell-3}
y_e^{\ell-3})/b_0$. Since all of the other labels are chosen to be
integer multiples of $b_0$, $y_e^0$ will be an integer. Also, because
of the order in which the labels are chosen, we have
$x_{v} < x_w < y_e^1 < \cdots < y_e^{\ell-3}$. Since
$a_1 \le a_2 \le b_1 \le \cdots  \le b_{\ell-3}$, and
(the negation of) the sum of the negative coefficients is less than
the sum of the positive coefficients, it follows that $y_e^0 > 0$, and
so $y_e^0 \in \mathbb{N}$. Thus $A \subseteq \mathbb{N}$, as required.
\end{proof}

%
Note: The exponent $2\ell+2$ could be reduced to $2\ell-1$, since for
each label the number of linear combinations involving it is
$\mathcal{O}(n^{2(\ell-1)})$ and the requirement that the labels be
totally ordered can be weakened to the requirement that each edge-set
is totally ordered, reducing the factor this contributes from $O(n^2)$
to $O(n)$.

\begin{cor} \label{ind-set-one}
Let $\LL$ be an equation
$c_1 x_1 + \cdots + c_{\ell} x_{\ell} = 0$, where $\ell \ge 3$
and the coefficients $c_i$ are all non-zero.
Let $G = (V,E)$ be a graph and $A = A_V \cup A_E$ the set
constructed in Lemma~\ref{construct-set}.

Then for any $k \in \mathbb{N}$, there is a one-to-one correspondence
between independent sets of $G$ of cardinality $k$ and the
$\LL$-free subsets of $A$ of cardinality $|A_E| + k$ which
contain all the elements of $A_E$.
\end{cor}

\begin{proof}
The corollary follows immediately from Lemma~\ref{construct-set}.
Given
an independent set $I$ of $G$, let $A_I = \{x_v : v \in I\}$,
then $A_I \cup A_E$ is
an $\LL$-free subset of $A$ of size
$|I|+|A_E|$;
since the vertex labels are distinct,
the $\LL$-free subsets
corresponding to independent sets $I_1 \neq I_2$ are distinct.
Also, given any $\LL$-free subset $S\cup A_E$ of $A$ of size
$|S|+|A_E|$, then $V_S = \{ v : x_v \in S\}$ is an independent set in $G$
of size $|S|$, and again we obtain a distinct
independent set for each such $\LL$-free subset.
\end{proof}

\begin{cor} \label{ind-set-two}
Let $\LL$ be an equation
$c_1 x_1 + \cdots + c_{\ell} x_{\ell} = 0$, where $\ell \ge 3$
and the coefficients $c_i$ are all non-zero.
Let $G = (V,E)$ be a graph and $A = A_V \cup A_E$ the set
constructed in Lemma~\ref{construct-set}.

Then for any $k \in \mathbb{N}$, $G$ contains an independent set of
cardinality $k$ if and only if $A$ contains an $\LL$-free subset of
cardinality $|A_E|+k$.
%
\end{cor}

\begin{proof}
By Corollary~\ref{ind-set-one}, it suffices to show that, if $A$
contains an $\LL$-free subset of cardinality $|A_E| + k$, then in fact
$A$ contains such a subset which includes all elements of $A_E$.  To
see this, let $A_1$ be an $\LL$-free subset of $A$ of size $|A_E| + k$
which does not contain all elements of $A_E$; we will show how to
construct an $\LL$-free subset of equal or greater size which does
have this additional property.

Suppose that $y_e^i\in A_E$ but $y_e^i \not \in A_1$, and suppose
that $e = (v,w)$.
Consider the labels $x_v,x_w$ of the endpoints of $e$.
By
Lemma~\ref{construct-set}(3) every non-trivial solution to $\LL$
involving $y_e^i$ must also involve
$x_v$ and $x_w$.
If one of these does not lie in $A_1$ we add $y_e^i$ to $A_1$ without
creating a solution to $\LL$. Otherwise, arbitrarily remove one of
$x_v$ and $x_w$ from $A_1$ and replace it with $y_e^i$.
Repeating this process, we obtain an $\LL$-free subset which contains
$A_E$ and is at least as large as $A_1$.
\end{proof}
%

\begin{thm} \label{np-completeness}
Let $\LL$ be a linear equation
$c_1 x_1 + \cdots + c_{\ell} x_{\ell} = 0$, where $\ell \ge 3$
and the coefficients $c_i$ are all non-zero.
Then $\LL$-\textsc{Free Subset} is strongly $\NP$-complete. If the
coefficients of $\LL$ are not all of the same sign, then the
input set $A$ can be restricted to be a subset of $\mathbb{N}$.
\end{thm}

\begin{proof}
We reduce \textsc{Independent Set} to $\LL$-\textsc{Free Subset}.
If $(G,k)$ is an instance of \textsc{Independent Set},
we construct (in polynomial
time) an instance $(A,|A_E|+k)$ of
$\LL$-\textsc{Free Subset}, with $A = A_V \cup A_E$, as in
Lemma~\ref{construct-set}.
Then by Corollary~\ref{ind-set-two},
$G$ has an independent set of size
$k$ if and only if $A$ has a subset of size $|A_E|+k$ with no
non-trivial solution
to $\LL$. The result follows.
\end{proof}

\subsection{Sub-equation free sets}
In order to prove the later results, it is helpful to have a slightly
stronger version of Lemma~\ref{construct-set}.

For a homogeneous equation $\LL$, we will call an equation obtained by
deleting a non-empty set of the terms of the equation
(that is, some variables and their coefficients) a proper sub-equation of $\LL$.
We will say that a set
$S$ is $\LL$-proper-sub-equation-free if and only if $S$ does not contain a
non-trivial solution to any of the proper sub-equations of $\LL$. If $S$
is also $\LL$-free, we will say it is $\LL$-sub-equation-free.

We now show that in Lemma~\ref{construct-set}, we can ensure that the
set $A$ constructed is $\LL$-proper-sub-equation-free.

\begin{lemma} \label{construct-set-sef}
Consider a linear equation
$c_1 x_1 + \cdots + c_{\ell} x_{\ell} = 0$, where $\ell \ge 3$
and the coefficients $c_i$ are all non-zero, and let
$\LL$ be an equivalent linear equation
$a_1 x_1 + a_2 x_2 + b_1 y_1 + \cdots + b_{\ell-3} y_{\ell-3} = b_0y_0$
in standard form.
Let $G = (V,E)$ be a graph. Then we can construct in polynomial time a
set $A \subseteq \mathbb{Z}$ with the following properties:
\begin{enumerate}
\item $A$ is the union of $\ell-1$ sets
$A_V, A_E^0 , A_E^1 , \ldots , A_E^{\ell-3}$, where
$A_V = \{x_v : v \in V\}$ and
$A_E^i = \{y_e^i : e \in E\}$ for each $i = 0, 1, \ldots , \ell-3$;
\item $|A| = |V| + (\ell-2)|E|$;
\item for every edge $e = (v,w)$, some permutation of
$(x_v,x_w, y_e^1 , \ldots , y_e^{\ell-3} , y_e^0)$ is a solution to
$\LL$; such a solution we call an edge-solution;
\item if $(z_1, \ldots , z_{\ell})$ is a non-trivial solution to
$\LL$, then $(z_1, \ldots , z_{\ell})$ is a permutation of
$(x_v,x_w, y_e^1 , \ldots , y_e^{\ell-3} , y_e^0)$ for some edge $e = (v,w)$;
\item $\max(A) = \mathcal{O}(|V|^{2\ell+2})$;
\item $A \subseteq \mathbb{N}$ unless all the coefficients of $\LL$
have the same sign;
\item $A$ is $\LL$-proper-sub-equation-free.
\end{enumerate}
\end{lemma}

\begin{proof}
The proof is very similar to the proof of Lemma~\ref{construct-set}.
In order to eliminate solutions to proper sub-equations, we have to
add extra inequalities to be satisfied by the members of $A$; that is,
for any sub-equation $\LL'$ with $k < \ell$ variables, and sequence
$\s = (Z_1,\ldots,Z_k) \in B^k$, we form the corresponding linear
combination $\lc(\s)$. If $\lc(\s)$ is identically zero,
this will correspond to a trivial solution to $\LL'$ and so can be
discarded. Otherwise, we form the reduced linear combination $\rlc(\s)$.
Then exactly as in Lemma~\ref{construct-set}, we can show that
$\rlc(\s)$ is not identically zero (in this case there are no edge-solutions
since $\LL'$ has at most $\ell-1$ variables). We then choose the set $A$
exactly as before, but now it will also satisfy the extra inequalities
which ensure that $A$ is $\LL$-proper-sub-equation-free.
The total number of inequalities is still
$\mathcal{O}(L^\ell)$, where $L$ is the number of variables.
\end{proof}

Using this construction in place of that in Lemma~\ref{construct-set},
we obtain the following $\NP$-completeness result:

\begin{cor} \label{np-completeness-sef}
Let $\LL$ be a linear equation
$c_1 x_1 + \cdots + c_{\ell} x_{\ell} = 0$, where $\ell \ge 3$
and the coefficients $c_i$ are all non-zero.
Then $\LL$-\textsc{Free Subset} is strongly $\NP$-complete. If the
coefficients of $\LL$ are not all of the same sign, then the
input set $A$ can be restricted to be a subset of $\mathbb{N}$.
Also the input set $A$ can be restricted to be
$\LL$-proper-sub-equation-free.
\end{cor}
%
%
\subsection{Inhomogeneous Equations}
We now consider inhomogeneous equations of the form
\[
c_1 x_1 + \cdots + c_{\ell} x_{\ell} = K,
\]
where $K \ne 0$. Clearly such an equation $\LL$ can only have a solution if
$K$ is divisible by $\gcd(c_1,\ldots,c_{\ell})$. In this case, we can
prove an analogue of Lemma~\ref{construct-set-sef}, provided we restrict
attention to tripartite graphs.

%
It is not sufficient, in order to prove the results of
Section~\ref{sec:proportion}, to require that
the set constructed is proper-sub-equation-free.
We will prove a stronger result which requires that there are no
solutions in which some variables take values from some fixed
$\LL$-free set.
\begin{lemma} \label{construct-set-inh}
Let $\LL$ be the linear equation
$c_1 x_1 + \cdots + c_{\ell} x_{\ell} = K$, where $\ell \ge 3$
and the coefficients $c_i$ and constant $K$ are all non-zero, with
$\gcd(c_1 , \ldots , c_{\ell})$ a divisor of $K$.
Let $G = (V,E)$ be a tripartite graph, with the vertex set $V$
partitioned into three independent sets $V_1, V_2, V_3$.
Also let $S'$ be a fixed $\LL$-free set.
Then we can construct in polynomial time a
set $A \subseteq \mathbb{Z}$ with the following properties:
\begin{enumerate}
\item $A$ is the union of $\ell-1$ sets
$A_V, A_E^0 , A_E^1 , \ldots , A_E^{\ell-3}$, where
$A_V = \{x_v : v \in V\}$ and
$A_E^i = \{y_e^i : e \in E\}$ for each $i = 0, 1, \ldots , \ell-3$;
\item $|A| = |V| + (\ell-2)|E|$;
\item for every edge $e = (v,w)$, some permutation of
$(x_v,x_w, y_e^1 , \ldots , y_e^{\ell-3} , y_e^0)$ is a solution to
$\LL$; such a solution we call an edge-solution;
\item if $(z_1, \ldots , z_{\ell})$ is a non-trivial solution to
$\LL$, then $(z_1, \ldots , z_{\ell})$ is a permutation of
$(x_v,x_w, y_e^1 , \ldots , y_e^{\ell-3} , y_e^0)$ for some edge $e = (v,w)$;
\item $\max(A) = \mathcal{O}(|V|^{2\ell})$;
\item $A \subseteq \mathbb{N}$ unless all the coefficients of $\LL$
have the same sign.
\item $A \cup S'$ has no solutions to $\LL$ except for those in $A$.
\end{enumerate}
\end{lemma}

\begin{proof}
First note that since the equation $\LL$ is inhomogeneous, there are no
trivial solutions. As in Lemma~\ref{construct-set}, we can construct a
set of inequalities which ensure that the vertex and edge labels are
all distinct, and there are no solutions to the equation apart from
the edge solutions.

Let $g = \gcd(c_1,\ldots,c_{\ell})$ so that there exist integers
$q_1 , \ldots q_{\ell}$ with
\[
c_1 q_1 + \cdots + c_{\ell} q_{\ell} = g .
\]
Also set $K' = K/g$.

In Lemma~\ref{construct-set} we associated the vertices with the first
two coefficients of the equation, however in this case, in order to
ensure that the dependent edge labels are integers, we will require
that the labels associated with coefficient $c_i$ are congruent to
$K'q_i$ modulo $c_1c_2c_3$.
This means that each vertex must be associated
with a particular coefficient.

As in Lemma~\ref{construct-set}, we will have one vertex label for
each vertex. For each edge there will be $\ell-3$ free edges
labels and one dependent edge label. Exactly as before,
the set of variables is
$B = B_V \cup B_E$, where
$B_E = B_E^0 \cup B_E^1 \cup \cdots \cup B_E^{\ell-3}$,
$B_V = \{X_v, v \in V\}$ and
$B_E^i = \{ Y_e^i : e \in E\}$ for each $i = 0 , 1 , \ldots , \ell-3$.
The variable $X_v$ is associated with the vertex $v$, and the
variables $Y_e^0, \ldots , Y_e^{\ell-3}$ with the edge~$e$.

We now consider the constraints that the variables must satisfy.

First, for each edge $e = (v,w)$, we have a set of variables
corresponding to the edge and its endpoints, that is
$\{X_v, X_w, Y_e^1, \ldots , Y_e^{\ell-3}, Y_e^0\}$. We will call such
a set an edge-set of variables.

For any sequence
$\s = (Z_1, \ldots, Z_{\ell}) \in B^{\ell}$ of
variables, we say that $\s$ is an edge-sequence if the set
$\{Z_1, \ldots, Z_{\ell}\}$ is an edge-set.

As before, we have the following inequalities:

Type~1: For any distinct pair $\s = (Z,Z') \in B^2$, we have the linear
combination $\lc(\s) = Z - Z' \ne 0$.

Type~2: For any sequence
$\s = (Z_1, \ldots, Z_{\ell}) \in (B \cup S')^{\ell}$ which is not an
edge-sequence, we have
$\lc(\s) = c_1 Z_1 + \cdots + c_{\ell} Z_{\ell} - K \ne 0$.

%
%

The variables in an edge-set must also satisfy the equation in some
order. To ensure this, we associate the three parts $V_1, V_2, V_3$ of
$V$ with coefficients $c_1, c_2, c_3$. Thus for any edge
$e = (v, w)$, the vertices $v, w$ will be in distinct parts,
so we associate these with corresponding coefficients, and the
dependent edge variable $y_e^0$ with the third of the coefficients
$c_1, c_2, c_3$. The free edge variables
$y_e^1 , \ldots , y_e^{\ell-3}$ are associated with
$c_4 , \ldots , c_{\ell}$. Thus if $v,w$ are in parts $p(v),p(w)$
respectively and $p(e)$ is the unique element of
$\{1,2,3\} \setminus \{p(v),p(w)\}$ we have the equation
\begin{equation} \label{eqn-inh-edge-equation}
c_{p(v)} X_v + c_{p(w)} X_w + c_{p(e)} Y_e^0 + c_4 Y_e^1 + \cdots + c_{\ell}
Y_e^{\ell-3} = K .
\end{equation}
Now in each of the linear combinations $\lc(\s)$,
we substitute for
each dependent edge variable $Y_e^0$ using these equations,
and collect terms to obtain a reduced
(rational) linear combination $\rlc(\s)$
of the vertex variables and free edge variables in the set
$B_V \cup B_E^1 \cup \cdots \cup B_E^{\ell-3}$.

Then, as in Lemma~\ref{construct-set}, none of the reduced combinations
is identically zero. Hence we can choose values for each of the vertex
labels and free edge labels so that all of the inequalities are
satisfied, and from these determine the dependent edge labels. We need
to check, however, that the dependent edge labels are actually
integers. As above, we choose the values so that a label associated
with coefficient $c_i$ is congruent to $K'q_i$ modulo $c_1c_2c_3$.

So suppose that the dependent edge label $y_e^0$ of some edge $e = (v,w)$
is associated with coefficient $c_{p(e)}$ (where $p(e) \in \{1,2,3\}$).
Then from Equation~\eqref{eqn-inh-edge-equation}
we have
\begin{align*}
K - (c_{p(v)} x_v + c_{p(w)} x_w + \sum_{i \ge 4} c_i
y_e^{i-3})
&\equiv K - \sum_{i \ne p(e)} c_i K' q_i \pmod{c_1c_2c_3}
\\
&\equiv K - K' (g - c_{p(e)} q_{p(e)}) \pmod{c_1c_2c_3}
\\
&\equiv  K' c_{p(e)} q_{p(e)} \pmod{c_1c_2c_3}
\\
&\equiv 0 \pmod{c_{p(e)}} .
\end{align*}
Hence $K - (c_{p(v)} x_v + c_{p(w)} x_w + \sum_{i \ge 4} c_i y_e^{i-3})$ is
divisible by $c_{p(e)}$ and so $y_e^0$, which is given by
\[
y_e^0 = (K - (c_{p(v)} x_v + c_{p(w)} x_w +
\sum_{i \ge 4} c_i y_e^{i-3}))/c_{p(e)} ,
\]
is an integer as required.

%
Note that the number $M$ of inequalities is
$\mathcal{O}(|A|^{\ell})$, so the elements of $A$ can be chosen to be
at most $\mathcal{O}(|A|^{\ell})=\mathcal{O}(|V|^{2\ell})$.

Finally we need to show that if not all of the coefficients
$c_1, \ldots , c_{\ell}$ have the same sign, then we can choose all
the labels to be positive.

Note first that for any $T > 0$, we can find a rational solution
\[
c_1 m_1 + \ldots + c_{\ell} m_{\ell} = K
\]
with each $m_i \ge T$. To see this, choose two coefficients
$c_I < 0$, $c_J > 0$ and choose each $m_i, i \ne I,J$ so that
$m_i \ge T$. Then if $\sum_{i \ne I,J} c_i m_i = N$, we have
$m_J = (K - N - c_I m_I)/c_J$ and it is clear that we can choose
$m_I, m_J$ both at least $T$.

Now if
\[
c_1 x_1 + \ldots + c_{\ell} x_{\ell} = K
\]
is a solution of $\LL$ and for some $I$, we have $|x_i - m_i| \le r$
for each $i \ne I$, then it follows that
\[
|x_I - m_I| \le (|K| + \sum_{i \ne I} |c_i| |x_i - m_i|)/|c_I| \le
(|K| + r \sum_{i \ne I} |c_i|)/|c_I| .
\]
Thus by setting
$R = (|K| + r\sum_i |c_i|)/(\min_i |c_i|)$, we have
$|x_I - m_I| \le R$. To ensure that all the $x_i$ are positive,
it suffices to ensure that $T > R$. We will exploit this by choosing
all the vertex and free edge labels to satisfy $|x_i-m_i|\leq r$. Then
the dependent edge labels are guaranteed to be positive.

Since there are $M$ inequalities to be satisfied, each label must
avoid at most $M$ values. For a label associated with the coefficient
$c_i$, we want to choose it to be positive and congruent
to $K'q_i$ modulo $c_1c_2c_3$.
Thus if $2r > (M+1)c_1c_2c_3$, then the
range $[m_i - r, m_i + r]$ includes at least $M+1$ integers congruent
to $K'q_i$ modulo $c_1c_2c_3$, so we can choose the label from among
these ensuring that it differs from $m_i$ by at most $r$. Hence the
dependent label will differ from one of $m_1, m_2, m_3$ by at most
$R$, and so will be positive provided $T > R$.
We can achieve this with a value of $T$ which is
$\mathcal{O}(M) = \mathcal{O}(|V|^{2\ell})$, hence the elements of $A$
can be positive and polynomially-sized.
\end{proof}

As above we obtain the following:

\begin{cor} \label{np-completeness-inh}
Let $\LL$ be a linear equation
$c_1 x_1 + \cdots + c_{\ell} x_{\ell} = K$, where $\ell \ge 3$
and the coefficients $c_i$ and constant $K$ are all non-zero, with
$\gcd(c_1 , \ldots , c_{\ell})$ a divisor of $K$.
Then $\LL$-\textsc{Free Subset} is strongly $\NP$-complete. If the
coefficients of $\LL$ are not all of the same sign, then the
input set $A$ can be restricted to be a subset of $\mathbb{N}$.
\end{cor}

\begin{proof}
The proof is similar to that of Theorem~\ref{np-completeness}, but the
input is restricted to tripartite graphs. We apply
Lemma~\ref{construct-set-inh} to obtain the
required set~$A$. Note that \textsc{Independent Set}
is $\NP$-complete for cubic planar graphs~\cite{garey-et-al-1976},
so in particular is
$\NP$-complete for tripartite graphs.
\end{proof}

\section{Approximation}
We now consider the complexity of determining the size of the largest
solution-free subset of a set. The material in this section is really
just a generalisation of the corresponding result given by Meeks and
Treglown~\cite{meeks-treglown}, so we follow their treatment closely.

We consider the following computational problem:

\begin{framed}
\noindent \textsc{Maximum $\LL$-Free Subset}\newline
\textit{Input:} A finite set $A \subseteq \mathbb{Z}$.\newline
\textit{Question:} What is the cardinality of the largest $\LL$-free
subset $A' \subseteq A$?
\end{framed}

If $\Pi$ is an optimisation problem such as the one above, then for
any instance $I$ of $\Pi$, we denote by
$\opt(I)$ the value of the optimal solution to $I$ (in this example
above, the cardinality of the largest solution-free subset).
An approximation algorithm $\mathcal{A}$ for $\Pi$ has
\emph{performance ratio $\rho$}
if for any instance $I$ of $\Pi$, the value $\mathcal{A}(I)$ given
by the algorithm $\mathcal{A}$ satisfies
\[
1 \leq \frac{\opt(I)}{\mathcal{A}(I)} \leq \rho.
\]
%

An algorithm $\mathcal{A}$ is a
\emph{polynomial-time approximation scheme (PTAS)} for $\Pi$ if it
takes as input an instance $I$ of $\Pi$ and a real number $\eps > 0$,
runs in time polynomial in the size of $I$ (but not necessarily in
$1/\eps$) and outputs a value $\mathcal{A}(I)$ such that
\[
1 \leq \frac{\opt(I)}{\mathcal{A}(I)} \leq 1 + \eps.
\]
We show in this section that for any linear equation $\LL$
with at least three variables, there can be no PTAS for the problem
\textsc{Maximum $\LL$-Free Subset} unless $\P = \NP$.

The complexity class $\APX$ contains all optimisation problems (whose
decision version belongs to $\NP$) which can be approximated within
some constant factor in polynomial time; this class includes problems
which do not admit a PTAS unless $\P=\NP$, so an optimisation problem
which is hard for $\APX$ does not have PTAS unless $\P=\NP$.
In order to show that a problem
is $\APX$-hard (and so does not admit a PTAS unless $\P=\NP$), it
suffices to give a PTAS reduction from another $\APX$-hard problem.

\begin{adef}
Let $\Pi_1$ and $\Pi_2$ be maximisation problems.  A \emph{PTAS
reduction from $\Pi_1$ to $\Pi_2$} consists of three polynomial-time
computable functions $f$, $g$ and $\alpha$ such that:
\begin{enumerate}
\item for any instance $I_1$ of $\Pi_1$ and any constant error
parameter $\eps$, $f$ produces an instance $I_2 = f(I_1,\eps)$
of $\Pi_2$;
\item if $\eps > 0$ is any constant and $y$ is any solution to
$I_2$ such that $\frac{\textrm{opt}(I_2)}{|y|} \leq \alpha(\eps)$,
then $x = g(I_1,y,\eps)$ is a solution to $I_1$ such that
$\frac{\textrm{opt}(I_1)}{|x|} \leq 1 + \eps$.
\end{enumerate}
\end{adef}

It was shown by Alimonti and Kann~\cite{alimonti-kann2000} that the problem
(called \textsc{Max IS-3}) of determining
the maximum size of an independent set in a graph of maximum degree 3
is $\APX$-hard.
\begin{lemma}
Let $\LL$ be a linear equation
$c_1 x_1 + \cdots + c_{\ell} x_{\ell} = K$, where $\ell \ge 3$
and the coefficients $c_i$ are all non-zero, and
$\gcd(c_1 , \ldots , c_{\ell})$ is a divisor of $K$.
Then there is a PTAS reduction from \textsc{Max IS-3} to
\textsc{Maximum $\LL$-Free Subset}.
\end{lemma}

\begin{proof}
The proof follows closely the proof of Lemma~7 in~\cite{meeks-treglown}.
We define the functions $f$, $g$ and $\alpha$ as follows.

First, we let $f$ be the function which, given an instance $G$ of
\textsc{Max IS-3} (where $G=(V,E)$) and any $\eps > 0$, outputs
the set $A = A_V \cup A_E \subseteq \mathbb{Z}$
described in  \protect{Lemma~\ref{construct-set}} in the homogeneous case or
\protect{Lemma~\ref{construct-set-inh}} (with $S' = \emptyset$) in the inhomogeneous case;
we know from these lemmas
that we can construct this set in polynomial time.

Next suppose that $B$ is an $\LL$-free subset in $A$.  We can
construct in polynomial time a set $\widetilde{B}$, with
$|\widetilde{B}| \geq |B|$, such that
\begin{enumerate}
\item $A_E \subseteq \widetilde{B}$; and
\item $\widetilde{B}$ is a maximal $\LL$-free subset of $A$.
\end{enumerate}
If $B$ fails to satisfy the first condition, we can use the method of
Corollary~\ref{ind-set-two} to obtain a set with this property,
and if the resulting set is not maximal we can add elements greedily
until this condition is met.  We now define $g$ to be the function
which, given an $\LL$-free set $B \subseteq A$ and any $\eps > 0$,
outputs the set $\{ v : x_v \in \widetilde{B} \setminus A_E\}$.

Finally, we define $\alpha$ by
$\alpha(\eps) = 1 + \frac{\eps}{6(\ell-2)+1}$.
Denote by $\opt(G)$ the
cardinality of the maximum independent set in $G$, and by $\opt(A)$
the cardinality of the largest $\LL$-free subset in $A$.  Note that
$\opt(A) = \opt(G) + (\ell-2)|E|$.  To complete the proof, it suffices to
show that whenever $B$ is an $\LL$-free subset in $A$ such
that $\frac{\opt(A)}{|B|} \leq \alpha(\eps) = 1 +
\frac{\eps}{6(\ell-2)+1}$, we have $\frac{\opt(G)}{|I|} \leq 1 +
\eps$, where $I = \{ v : x_v \in \widetilde{B} \setminus A_E\}$.
Observe that

\begin{align*}
\frac{\opt(A)}{|B|} &\leq 1 + \frac{\eps}{6(\ell-2)+1} \\
\Rightarrow \frac{\opt(A)}{|\widetilde{B}|} &\leq 1 +
\frac{\eps}{6(\ell-2)+1} \\
\Rightarrow \frac{(\ell-2)|E| + \opt(G)}{(\ell-2)|E| + |I|} & \leq 1 +
\frac{\eps}{6(\ell-2)+1} \\
\Rightarrow \frac{(\ell-2)|E| + \opt(G)}{|I|} & \leq \left(1 +
\frac{\eps}{6(\ell-2)+1}\right)\left(\frac{(\ell-2)|E| + |I|}{|I|}\right) \\
\Rightarrow \frac{\opt(G)}{|I|} & \leq \left(1 +
\frac{\eps}{6(\ell-2)+1}\right)\frac{(\ell-2)|E|}{|I|} + \left(1 +
\frac{\eps}{6(\ell-2)+1}\right) - \frac{(\ell-2)|E|}{|I|} \\
& = \frac{\eps}{6(\ell-2)+1}\frac{(\ell-2)|E|}{|I|} + 1 +
\frac{\eps}{6(\ell-2)+1}.
\end{align*}
Since we know that $|E| \leq \frac{3|V|}{2}$ and, by our assumptions
on maximality of $\widetilde{B}$ and hence $I$, we also know that $|I|
\geq \frac{|V|}{4}$, it follows that $\frac{|E|}{|I|} \leq 6$.  We can
therefore conclude that
\[
\frac{\opt(A)}{|B|} \leq 1 + \frac{\eps}{6(\ell-2)+1}
\Rightarrow \frac{\opt(G)}{|I|} \leq 6(\ell-2)\frac{\eps}{6(\ell-2)+1} + 1
+ \frac{\eps}{6(\ell-2)+1} = 1 + \eps,
\]
as required.
\end{proof}

As a corollary, we obtain the following:
\begin{thm}
Let $\LL$ be a linear equation with at least three variables.
Then \textsc{Maximum $\LL$-Free Subset} is $\APX$-hard.
\end{thm}
\section{$\LL$-free subsets with a given proportion of
elements}\label{sec:proportion}
In this section we consider the problem of deciding if a set has
a solution-free subset containing a given proportion of its elements.

Let $\LL$ be the equation
\[
c_1 x_1 + \cdots + c_{\ell} x_{\ell} = K
\]
and let $C = \sum_i^{\ell} c_i$.

If $C \ne 0$ and $C$ divides $K$, let $S_{\LL} = \{K/C\}$ (otherwise
$S_{\LL}$ is empty). Then note that $S_{\LL}$ is an integer set
that has no non-empty subset which is $\LL$-free.
For this reason we need to exclude this value.
%
Let $\mathcal{C}(\LL)$ denote the set of all positive reals
$\lambda$ so that every finite set
$Z\subseteq \mathbb{Z}\setminus S_{\LL}$  contains
an $\LL$-free subset of size greater than $\lambda |Z|$.

Define
\[
\kappa (\LL)=\sup (\mathcal{C}(\LL)).
\]
It follows from the results of Erd\H{o}s~\cite{erdos1965}
and Eberhard, Green and Manners~\cite{eberhard-green-manners-2014}
that for the sumfree equation
$\LL : x + y = z$, we have $\kappa(\LL) = 1/3$.

Given any linear equation $\LL$  and a constant $\eps$ satisfying
$0<\eps<1$, we define the following problem.
\begin{framed}
\noindent $\eps$-$\LL$-\textsc{Free Subset}\newline
\textit{Input:} A finite set $A \subseteq \mathbb{Z} \setminus S_{\LL}$.\newline
\textit{Question:} Does there exist an $\LL$-free subset $A' \subseteq
A$ such that $|A'|\geq \eps|A|$?
\end{framed}
Note that if $\eps \leq \kappa (\LL)$, then it follows from the
definition of $\kappa(\LL)$ that every instance of
$\eps$-$\LL$-\textsc{Free Subset} is a yes-instance, and so the
problem is trivially in $\P$ in this case.
%
%
%
%
\subsection{Hardness of $\eps$-$\LL$-Free Subset}\label{62}

In this section we show that
$\eps$-$\LL$-\textsc{Free Subset} is strongly $\NP$-complete whenever
$\kappa (\LL) < \eps < 1$.
Given a set $X \subseteq \mathbb{Z}$ and $y \in
\mathbb{N}$, we write $yX$ as shorthand for $\{yx: x \in X\}$,
and $X + y$ for $\{x + y : x \in X\}$.

Recall that for a homogeneous equation $\LL$,
we say that a set
$S$ is $\LL$-sub-equation-free if and only if $S$ does not contain a
non-trivial solution to $\LL$ or any of its sub-equations.

\begin{lemma}\label{sub-equation-free}
Let $\LL$ be a translation invariant homogeneous linear equation,
and let $S$ be an $\LL$-free set of integers.

Then for all except a finite number of positive integers $\alpha$, the
set $S+\alpha$ is $\LL$-sub-equation-free.
\end{lemma}

\begin{proof}
Let $\LL$ be the equation $c_1 x_1 + \cdots + c_{\ell} x_{\ell} = 0$,
and let $S = \{s_1, \ldots , s_n\}$ be an $\LL$-free set.
We will say a subsequence $(c_{q_1}, \ldots , c_{q_k})$
of
$(c_1, \ldots , c_{\ell})$ has non-zero sum
if $c_{q_1} + \cdots + c_{q_k} \ne 0$.
We choose $\alpha$ to be any positive integer which is not equal to
$-(c_{q_1} s_{r_1} + \cdots + c_{q_k} s_{r_k})/(c_{q_1} + \cdots +
c_{q_k})$ for any subsequence $c_{q_1}, \ldots , c_{q_k}$ with
non-zero sum and
any $(s_{r_1}, \ldots , s_{r_k}) \in S^k$.

Then $S + \alpha$ is $\LL$-sub-equation-free.
For suppose that $\LL'$: $c_{q_1} x_{q_1} + \cdots + c_{q_k} x_{q_k} = 0$ is a
sub-equation of $\LL$, and that
$(s_{r_1}+\alpha, \ldots , s_{r_k}+\alpha) \in (S+\alpha)^k$ is a
solution to $\LL'$. We will show that the solution is trivial.
First, $(c_{q_1}, \ldots , c_{q_k})$ cannot have non-zero sum, for then
\[
c_{q_1} (s_{r_1}+\alpha) + \cdots + c_{q_k} (s_{r_k}+\alpha) =
c_{q_1} s_{r_1} + \cdots + c_{q_k} s_{r_k} +
\alpha(c_{q_1} + \cdots + c_{q_k}) \ne 0
\]
by the definition of $\alpha$, a contradiction.
So $c_{q_1} + \cdots + c_{q_k} = 0$, which means that
\[
c_{q_1} s_{r_1} + \cdots + c_{q_k} s_{r_k}
=
c_{q_1} (s_{r_1}+\alpha) + \cdots + c_{q_k} (s_{r_k}+\alpha)
=
0 .
\]
Since $\LL$ is
translation invariant, the sum of the remaining coefficients,
$\sum_{i \not\in \{q_1,\ldots,q_k\}} c_i$, is also zero. Thus if $s$
is any element of $S$, we have
\[
c_{q_1} s_{r_1} + \cdots + c_{q_k} s_{r_k} +
\sum_{i \not\in \{q_1,\ldots,q_k\}} c_i s = 0 .
\]
Since $S$ is $\LL$-free, this solution of $\LL$ must be trivial, which means
that the total coefficient of each $s_i$ in
$c_{q_1} s_{r_1} + \cdots + c_{q_k} s_{r_k}$ must be zero.
Hence the total coefficient of each $s_i + \alpha$ in
$c_{q_1} (s_{r_1}+\alpha) + \cdots + c_{q_k} (s_{r_k}+\alpha)$ is
zero, so the solution $(s_{r_1}+\alpha, \ldots , s_{r_k}+\alpha)$ of
$\LL'$ is trivial, as required.
\end{proof}

We now show how to extend an $\LL$-free set (provided certain
conditions are met). This will be useful to create sets whose
largest $\LL$-free set is of some desired size.

We say that an injective function $f : \mathbb{Z} \rightarrow \mathbb{Z}$ is
solution-preserving for $\LL$ if the following holds:
$(f(x_1),\ldots,f(x_{\ell}))$ is a solution to $\LL$ if and only if
$(x_1,\ldots,x_{\ell})$ is a solution to $\LL$.

\begin{lemma}\label{extend}
Let $\LL$ be the linear equation
$c_1 x_1 + \ldots + c_{\ell} x_{\ell} = K$ with $\ell \ge 3$.
Let $S$ be a fixed subset of $\mathbb{Z} \setminus S_{\LL}$, $S'
\subseteq S$ a fixed $\LL$-free subset of $S$ and $A$ be a set of
integers with the following properties.
\begin{itemize}
\item If $K = 0$ and
$\LL$ is translation invariant, then $S'$ is $\LL$-sub-equation-free and $0 \not\in S$.
\item If $K = 0$, then $A$ is $\LL$-proper-sub-equation-free.
\item If $K \ne 0$, then $A \cup S'$ has no solutions to $\LL$ except for those in $A$.
\end{itemize}
Let $r$ and $t$ be positive integers.
%
%
Then we can choose solution-preserving functions $f_1,\ldots,f_r$ and a
set $T = \{u_1 , \ldots , u_t \}$ comprising $t$ distinct strictly
positive integers such that
the
sets $A$, $f_i(S)$ for each $i = 1,\ldots, r$, and $T$ are all mutually
disjoint and such that
\[
A \cup \bigcup_{i=1}^r f_i(S') \cup T
\]
has no non-trivial solutions to $\LL$ except those in $A$.
In particular, if $A' \subseteq A$ is $\LL$-free, then
$A' \cup \bigcup_{i=1}^r f_i(S') \cup T$ is $\LL$-free.
Also the functions $f_1,\ldots,f_r$ are polynomially bounded and the elements of
$T$ need be no larger than $\mathcal{O}(|A| + r + t)^{\ell}$.
\end{lemma}

\begin{proof}
We will use a similar technique to that used in
Lemma~\ref{construct-set}, that is, we will use a variable
corresponding to each integer that needs to be chosen, and obtain a set
of inequalities in these variables. Then choosing values of the
variables to satisfy all of the inequalities will give the result.

Let $S = \{ s_1, \ldots , s_k\}$, and assume that
$S' = \{ s_1, \ldots , s_{k'}\}$.
First consider the case when the equation is homogeneous, so that
$K=0$. In this case we will use the functions $f_i(s) = d_i s$, where
$d_1 , \ldots, d_r$ are integers to be chosen.

We will take a set $D = \{D_1, \ldots , D_r\}$ of variables, with
variable $D_i$ corresponding to integer $d_i$, and a set
$U = \{U_1, \ldots , U_t\}$ of variables with $U_i$ corresponding to element
$u_i$. We will also write $s_{ij}$ for $d_i s_j$ and take a variable
$S_{ij}$ corresponding to $s_{ij}$; the variable $S_{ij}$ of course
depends on $D_i$ since $S_{ij} = D_i s_j$. Write
$\Sigma = \{S_{ij} : 1 \le i \le r, 1 \le j \le k \}$ and
$\Sigma' = \{S_{ij} : 1 \le i \le r, 1 \le j \le k' \}$.

Now for any distinct pair
$\s = (Z,Z') \in (A \cup \Sigma \cup U)^2 \setminus A^2$, we
form the combination $\ac(\s) = Z - Z'$.
Note that we require that
all the new elements are distinct, so $Z$ and $Z'$ range over
$A \cup \Sigma \cup U$ and not merely $A \cup \Sigma' \cup U$.

Also for any sequence
$\s = (Z_1, \ldots , Z_{\ell}) \in (A \cup \Sigma' \cup U)^{\ell}
\setminus A^{\ell}$,
we form the combination
$\ac(\s) = c_1 Z_1 + \ldots + c_{\ell} Z_{\ell}$.
Note that we only include sequences with at least one variable from
$\Sigma' \cup U$, and in this case we do not consider variables from
$\Sigma \setminus \Sigma'$.

For each sequence $\s$, we substitute $S_{ij} = D_i s_j$ for any
variables in $\Sigma$ to obtain the reduced combination $\rac(\s)$.

We will say that $\rac(\s)$ is identically zero if both the
constant part (which is a linear combination of elements from $A$)
and the total coefficient of each variable, are zero.
We will say that $\ac(\s)$ is trivial if the total coefficient of each
element of
$A \cup \Sigma \cup U$ in $\ac(\s)$ is zero. Note this is stronger than saying
that $\ac(\s)$ is identically zero, because we require that the total
coefficient of each element of $A$ is zero, not merely that the total
constant term is zero. The combination $\ac(\s)$ is trivial if and only
if $\s$ corresponds to a trivial solution of $\LL$, that is,
the total coefficient of each value is zero.

We claim that if $\rac(\s)$ is identically zero, then $\ac(\s)$ is
trivial. Equivalently, we must show that either $\rac(\s)$ is not
identically zero, or $\ac(\s)$ is trivial.

For distinct pairs
$\s = (Z,Z') \in (A \cup \Sigma \cup U)^2 \setminus A^2$ this is clear,
since the only case where the terms in $\ac(\s)$ could cancel after
substitution is if
$Z = S_{ij}$ and $Z' = S_{i'j'}$. But then
$Z - Z' = S_{ij} - S_{i'j'}$ which reduces to
$\rlc(\s) =s_j D_i - s_{j'}D_{i'}$.
If this is identically zero then we must have $s_j = s_{j'}$ so
$\rlc(\s) = s_j (D_i - D_{i'})$. Since $0 \not\in S$ (if $\LL$ is
translation invariant, this is by assumption, otherwise because
$S_{\LL} = \{0\}$), this means that
$D_i = D_{i'}$ also, contradicting the assumption that $Z,Z'$ are
distinct.

Now consider any sequence
$\s = (Z_1, \ldots , Z_{\ell}) \in (A \cup \Sigma' \cup U)^{\ell}
\setminus A^{\ell}$,
so that
$\ac(\s) = c_1 Z_1 + \ldots + c_{\ell} Z_{\ell}$.
Note that $\ac(\s)$ is the sum of three parts, an
$A$-part, which is an integer formed from a linear combination of
elements of~$A$, a $\Sigma'$-part
which is a linear combination of variables in $\Sigma'$, and a $U$-part
which is a linear combination of variables in $U$. Similarly $\rac(\s)$
is the sum of an $A$-part, a $D$-part and a $U$-part. Note that since the
substitution only affects the variables in $\Sigma'$, the $A$-parts
of $\ac(\s)$ and $\rac(\s)$ are equal, and similarly for the $U$-parts.
For any $D$-variable $D_i$, the coefficient of $D_i$ in $\rac(\s)$
is a linear combination of
the form $c_{q_1} s_{r_1} + \ldots + c_{q_k} s_{r_k}$, where each
$s_{r_j} \in S'$.

%
%
%
Now if $\rac(\s)$ is identically zero, its $A$-part must be zero.
This gives a solution to a sub-equation of $\LL$, which must be
trivial since $A$ is $\LL$-proper-sub-equation-free. Thus for each element of
the sequence $\s$ which comes from the set $A$, its total coefficient
in $\rac(\s)$, and therefore also in $\ac(\s)$, must be zero.

Also if $\rac(\s)$ is identically zero, its $U$-part must be
identically zero, hence the same is true for $\ac(\s)$, that is, every
$U$-variable in $\s$ has total coefficient zero in $\ac(\s)$.
Now there are two cases, depending on whether or not
$\LL$ is translation invariant.

First suppose that $\LL$ is not translation invariant, so that
$\ac(\s)$ cannot be trivial. Then we must show that $\rac(\s)$ is not
identically zero. So suppose that it is.
Pick a
fixed $s \in S'$, and in $\rac(\s)$, replace each element of $A$ with
the value $s$, and set each $D$-variable (temporarily) to~$1$ and
each $U$-variable (temporarily) to~$s$. Since $\rac(\s)$ is identically zero,
then from above the total coefficients of its
$A$- and $U$-parts are both zero, and this results in a solution to $\LL$ in
$S'$, contradicting the assumption that $S'$ is $\LL$-free.

Now suppose that $\LL$ is translation invariant.
In this case, by assumption, $S'$ is $\LL$-sub-equation-free.

Suppose that $\rac(\s)$ is identically zero. Then we must show that
$\ac(\s)$ is trivial. From above, the total coefficient in $\ac(\s)$ of
each element of $A$ and each $U$-variable is zero, so
to show that $\ac(\s)$ is trivial, we need to show that the total
coefficient in $\ac(\s)$ of each $\Sigma'$-variable is also zero. We do this
as follows. Since $\rac(\s)$ is identically zero, the total coefficient
of each $D$-variable in $\rac(\s)$ is zero.

Consider the coefficient of
some $D$-variable $D_i$. This coefficient is a linear combination of
the form $c_{q_1} s_{r_1} + \ldots + c_{q_k} s_{r_k}$, where each
$s_{r_j} \in S'$.
As the coefficient of $D_i$ is $0$,
this gives a solution to a sub-equation of $\LL$, so
since $S'$ is $\LL$-sub-equation-free, this solution is trivial and so
the total coefficent of each $s_j$ in
$c_{q_1} s_{r_1} + \ldots + c_{q_k} s_{r_k}$ is zero. But this is just the
coefficient of $s_j D_i = S_{ij}$ in $\ac(\s)$, hence $\ac(\s)$ is
trivial, as required.

Hence whenever $\ac(\s)$ is not trivial, $\rac(\s)$ is not identically
zero. We form a set of inequalities by including the inequality
$\rac(\s) \ne 0$ whenever $\ac(\s)$ is non-trivial.

Now as in Lemma~\ref{construct-set}, we choose values
$d_i = \val(D_i)$ and $u_i = \val(U_i)$ for each of the
variables in $D \cup U$, from which those in $\Sigma$ follow by
setting $s_{ij} = d_i s_j$. We can choose the values of the $D$-variables
and $U$-variables to ensure that the
value $\val(\rac(\s))$ of $\rac(\s)$ is non-zero. But since clearly
$\val(\rac(\s) = \val(\ac(\s))$, this ensures that there are no
non-trivial solutions to $\LL$ in $A \cup \bigcup_{i=1}^r d_i S' \cup T$
except those already present in $A$.
In particular, if $A' \subseteq A$ is
$\LL$-free, then so is $A' \cup \bigcup_{i=1}^r d_i S' \cup T$.

We now consider the inhomogeneous case, where $K \ne 0$. This case is
simpler because there are no trivial solutions. As before we choose
functions $f_1, \ldots , f_r$. These depend on $r$ integers
$d_1, \ldots , d_r$ which we will choose as above using $r$ variables
$D_1, \ldots , D_r$.

Let $C = c_1 + \cdots + c_{\ell}$ be the sum of the coefficients. The
solution-preserving functions functions $f_i$ depend on whether $C$ is
zero or not. If
$C = 0$, let $f_i$ be the function defined by $f_i(s) = s + d_i$. If
$C \ne 0$, define $f_i$ by $f_i(s) = (d_iC+1)s - d_iK$. Now as above
set $s_{ij} = f_i(s_j)$, and let $S_{ij}$ be a variable corresponding
to $s_{ij}$, so that $S_{ij} = s_j + D_i$ if $C = 0$ and
$S_{ij} = (D_iC+1)s_j - D_iK = (s_jC - K)D_i + s_j$ otherwise. As above
$\Sigma = \{S_{ij} : 1 \le i \le r, 1 \le j \le k \}$ and
$\Sigma' = \{S_{ij} : 1 \le i \le r, 1 \le j \le k' \}$.

Now as before, for any distinct pair
$\s = (Z,Z') \in (A \cup \Sigma \cup U)^2 \setminus A^2$, we
form the combination $\ac(\s) = Z - Z'$,
and for any sequence
$\s = (Z_1, \ldots , Z_{\ell}) \in (A \cup \Sigma' \cup U)^{\ell}
\setminus A^{\ell}$,
we form the combination
$\ac(\s) = c_1 Z_1 + \ldots + c_{\ell} Z_{\ell} - K$.

As before we substitute for each $S_{ij}$ to obtain the reduced
combinations $\rac(\s)$.
Note that the constant part of $\rlc(z)$ is
now a linear combination of elements from $A$ and $S$.
We need to establish that $\rac(\s)$ is not
identically zero. As before, for a distinct pair $\s = (Z,Z')$, a
problem could only arise if $Z = S_{ij}$ and $Z' = S_{i'j'}$. Then
$\lc(\s) = Z - Z'$. If $C = 0$ then
$\rlc(\s) = s_j + D_i - (s_{j'} + D_{i'})$, which can only be
identically zero if $s_j = s_{j'}$ and $D_i = D_{i'}$, so $Z$ and $Z'$
are the same variable. If $C \ne 0$ then
$\rlc(\s) = (s_jC-K)D_i +s_j - ((s_{j'}C-K)D_{i'} +s_{j'})$. If this
is identically zero, we must have $s_j = s_{j'}$, and then
$\rlc(\s) = (s_jC-K)(D_i - D_{i'})$. Since $K/C \not\in S$, this can
only be identically zero if $D_i = D_{i'}$ but then $Z$ and $Z'$ are
the same variable, giving a contradiction.

So suppose that $\s$ is a sequence
$(Z_1, \ldots , Z_{\ell}) \in (A \cup \Sigma' \cup U)^{\ell}
\setminus A^{\ell}$. Note that some $Z_i$ is in $\Sigma' \cup U$.
We substitute for each $\Sigma'$-variable to obtain $\rac(\s)$.
Recall that these substitions are either
$S_{ij} = s_j + D_i$ or $S_{ij} = (s_jC - K)D_i + s_j$.
Then $\rac(\s)$ has a constant part, a $D$-part
and a $U$-part. The constant part is of the form
$c_{q_1} x_{q_1} + \ldots + c_{q_k} x_{q_k} - K$, where each
$x_{q_j}$ is in $A \cup S'$.
Then observe that the $U$-part is
$\sum_{i \not\in \{q_1, \ldots , q_k\}} c_i Z_i$, where each
$Z_i \in U$. Also either some $x_{q_j}$ is in $S'$, or the $U$-part
has at least one term. If $\rac(\s)$ is identically zero, then the constant part
must be zero and the $U$-part identically zero. But then setting
(temporarily) each $U$-variable equal to some fixed element of $S'$
gives a solution to $\LL$ in $A \cup S'$
involving at least one term in $S'$, so this solution
%
is not in $A$, a contradiction.

Thus $\rac(\s)$ is not identically zero, so we can choose
values of the $D$- and $U$-variables to ensure that every $\ac(\s)$ is non-zero,
and then there are no additional solutions to $\LL$
in $A \cup \bigcup_{i=1}^r f_i(S') \cup T$.

The
number of elements in $A \cup D \cup U$ is $|A| + r + t$, hence the
number of inequalities is $\mathcal{O}(|A| + r + t)^{\ell}$. Hence
each value chosen has to avoid up to this many values, so we can
choose each element of $D\cup U$ to be no larger than
$\mathcal O(|A|+r+t)^\ell$, with the elements of $U$ being positive.

\end{proof}

We can now state and prove the $\NP$-completeness result for
$\eps$-$\LL$-\textsc{Free Subset}. Again this result is a
generalisation of one given by Meeks and Treglown \cite{meeks-treglown}
and the proof closely follows theirs.
\begin{thm}\label{thmc}
Let $\LL$ be a linear equation with at least three
variables, and let $\eps$ be a rational number satisfying
$\kappa (\LL) < \eps < 1$.
Then $\eps$-$\LL$-\textsc{Free Subset} is strongly $\NP$-complete.
\end{thm}

\begin{proof}
Let $\LL$ be the linear equation
$c_1 x_1 + \ldots + c_{\ell} x_{\ell} = K$ with $\ell \ge 3$.

By definition of $\kappa(\LL)$, there is
a set $S\subseteq \mathbb{Z} \setminus S_{\LL}$ such that the largest $\LL$-free
subset $S'$ of $S$ has size precisely $\eps'|S|$ where
$\kappa(\LL)\leq\eps'<\eps$.
If $\LL$ is homogeneous
and translation invariant, then by Lemma~\ref{sub-equation-free} we
can find an integer $\alpha$ such that $S'+\alpha$ is
$\LL$-sub-equation-free and $0 \not\in S + \alpha$.
Clearly the largest $\LL$-free subset of
$S+\alpha$ has size $\eps'|S+\alpha|$. Thus in the case when $\LL$
is homogeneous and translation invariant we can assume
(by replacing $S$ by $S+\alpha$ and $S'$ by $S'+\alpha$) that
$S'$ is $\LL$-sub-equation-free and $0 \not\in S$.

It is shown in \cite{meeks-treglown}
that $\eps$-$\LL$-\textsc{Free Subset} belongs to $\NP$.
To show that the problem is $\NP$-complete, we describe a reduction from
$\LL$-\textsc{Free Subset}%
, shown to be $\NP$-complete in
Theorems~\ref{np-completeness} and
Corollary~\ref{np-completeness-inh}. We can also, by
Corollary~\ref{np-completeness-sef} and Lemma~\ref{construct-set-inh},
assume that the elements of $A$ are bounded by a polynomial in
$|A|$, and that (i) if $K = 0$ then $A$ is
$\LL$-proper-sub-equation-free and (ii)
if $K \ne 0$ then $A \cup S'$ has no solution to $\LL$ except for
those in $A$.

Suppose that $(A,k)$ is such an instance of $\LL$-\textsc{Free Subset} where
$A\subseteq \mathbb{N}$.
We will define a set $B \subseteq \mathbb{Z}$ such that $B$ has an
$\LL$-free subset of size at least $\eps|B|$ if and only if $A$ has an
$\LL$-free subset of size $k$.
Note that we may assume that $k \le |A|$, since otherwise the
instance $(A,k)$ is a ``no''-instance and we can take $B$ to be any
$\LL$-free set.

Let $a = |A|$, and set
\[
r = \max \left( 0,\left\lceil \frac{k - \eps a}{(\eps - \eps')|S|}
\right\rceil \right) ,
\]
and note that $r = \mathcal{O}(|A|)$.
Let $k^* = k + r|S'| = k + r\eps'|S|$ and $a^* = a + r|S|$.
Then by definition of $r$,
$r \eps |S| - r \eps' |S| \ge k - \eps a$, so that
$k + r \eps' |S| \le \eps a + r \eps |S| = \eps(a + r|S|)$,
or $k^* \le \eps a^*$.

Now set
\[
t = \left\lceil \frac{\eps a^* - k^*}{1 - \eps} \right \rceil,
\]
and again note that $t = \mathcal{O}(|A|)$.
Then
\[
t \ge \frac{\eps a^* - k^*}{1 - \eps} \ge t-1,
\]
so
\[
t(1-\eps) \ge \eps a^* - k^* \ge (t-1)(1-\eps),
\]
and so
\[
k^* + t \ge \eps (a^* + t) \ge k^* + t - (1-\eps).
\]
Hence $\lceil \eps (a^* + t) \rceil = k^* + t$, or
$\lceil \eps (|A| + r|S| + t) \rceil = k + r|S'| + t$.

By Lemma~\ref{extend}, we can construct, in
time polynomial in $\size{A}$, a set
$B = A \cup \bigcup_{i=1}^r f_i(S) \cup T$, where the functions $f_i$
are solution-preserving, such that
$|B| = |A| + r|S| + t$ and such that
$A \cup \bigcup_{i=1}^r f_i(S') \cup T$ has no non-trivial solutions
to $\LL$ apart from those in $A$.

Then we claim that
$B$ has an $\LL$-free subset of size at least $\eps |B|$ if and only
if $A$ has an $\LL$-free subset of size $k$.

To show this, first suppose that $B$ has an $\LL$-free subset $B'$ of size
at least $\eps |B|$. Then $B'$ has size at least
$\lceil \eps |B| \rceil = \lceil \eps (|A| + r|S| + t) \rceil = k +
r|S'| + t$. Clearly
the largest $\LL$-free set in  $f_i(S)$ is of size $|S'|$.
Hence at most $r|S'|+t$ of the elements of $B'$ can be in
$\bigcup_{i=1}^r f_i(S) \cup T$ so at least $k$ are in $A$, as
required.

On the other hand, suppose that $A$ has an $\LL$-free subset $A'$ of
size $k$. Then by Lemma~\ref{extend}, the set
$A' \cup \bigcup_{i=1}^r f_i(S') \cup T$ is $\LL$-free, and this set
has size
$k + r|S'| + t = \lceil \eps (|A| + r|S| + t) \rceil = \lceil \eps |B|
\rceil$.

Hence $B$
is a yes-instance for $\eps$-$\LL$-\textsc{Free Subset} if and only if
$(A,k)$ is a yes-instance for $\LL$-\textsc{Free Subset}.
\end{proof}
\section{Counting the number of $\LL$-free subsets}\label{sec:counting}
We move on to consider the counting version of $\LL$-\textsc{Free Subset}.

\medskip
\begin{framed}
\noindent $\#\LL$-\textsc{Free Subset}\newline
\textit{Input:} A finite set $A \subseteq \mathbb{Z}$. \newline
\textit{Output:} The number of $\LL$-free subsets of $A$.
\end{framed}
We first consider the homogeneous case. Here we use two different
constructions, with one construction covering equations with three variables and the other equations with four or more variables.

Each construction requires another modified version of Lemma~\ref{construct-set}.
\begin{lemma}\label{lem:counting3.2}
Consider a linear equation $c_1x_1+\cdots+c_3 x_3=0$, where the coefficients $c_i$ are all non-zero, and let $\LL$ be an equivalent linear equation $a_1x_1+a_2x_2=b_0y_0$ in standard form.
Let $G=(V,E)$ be a graph and let $r\in \nats$.
Then we can construct in polynomial time, a set $A\subseteq \ints$ with the following properties.
\begin{enumerate}
\item $A$ is the union of three sets
$A_V,\,A_E$ and $U$, where
$A_V=\{x_v:v\in V\}$, $A_E=\{y_e:e\in E\}$ and
$U=\{u_{v,e,i}:v\in V,\, e=vw\in E,\,1\leq i\leq r\}$;
\item $|A| = |V| + (1+2r)|E|$;
\item for every edge $e=vw$, some permutation of $(x_v,x_w,y_e)$ is a solution to $\LL$; such a solution we call an edge solution;
\item for every edge $e=vw$ and every integer $i$ with $1\leq i\leq r$,
some permutation of $(u_{v,e,i},u_{w,e,i},y_e)$ is a solution to $\LL$;
\item if $(z_1,z_2,z_3)$ is a non-trivial solution to $\LL$, then $(z_1,z_2,z_3)$ is a permutation of
$(x_v,x_w,y_e)$ for some edge $e=vw$ or of $(u_{v,e,i},u_{w,e,i},y_e)$ for some edge $e=vw$ and integer $i$ with $1 \leq i \leq r$;
\item $\max(A)=\mathcal{O}(|V|^{12}r^2)$;
\item $A\subseteq \nats$ unless all the coefficients $c_1,c_2,c_3 $ have the same sign.
\end{enumerate}
\end{lemma}

\begin{proof}
Apply Lemma~\ref{construct-set} to produce a set $A'$ with $A'_V=\{x'_v:v\in V\}$ and $A'_E=\{y_e:e\in E\}$.
We will modify $A'$ by multiplying each element by a large positive integer $N$ and then extend it by adding, for each edge $(v,w)$, $r$ new elements close to each of the vertex labels $x_v=Nx'_v$ and $x_w=Nx'_w$.

We first choose a strictly positive integer $N_{e,i}$ for each edge
$e$ and integer $i$ with $1\leq i \leq r$, so that the set
\[ S = \{ a_2N_{e,i},-a_1N_{e,i}: e\in E,\,1 \leq i \leq r\} \cup \{0\}\]
has $2r|E|+1$ distinct elements and
if $\{z_1,z_2,z_3\} \subseteq S$ and some permutation of
$(z_1,z_2,z_3)$ is a solution to $\LL$, then there exists an edge $e$
and integer $i$ with $1\leq i\leq r$ such that $\{z_1,z_2,z_3\}
\subseteq \{a_2N_{e,i}, -a_1N_{e,i},0\}$.

We may choose the integers $N_{e,i}$ greedily in any order. When one
of them is chosen, there are strictly fewer than $\mathcal{O}(|E|^2
r^2)$ linear equations that it must not satisfy.
Thus they may be chosen so that $N'=\max_{e\in E,\ 1\leq i \leq r}N_{e,i}$
satisfies $N'=\mathcal{O}(|E|^2 r^2)$.

Now let $c=\max\{|c_1|,|c_2|,|c_3 |\}$ and $N=3c^2N'+1$. For each
$v\in V$, let $x_v=Nx'_v$, and for each $e\in E$ let $y_e=Ny'_e$. Form
$A_V$ and $A_E$ by setting
$A_V=\{x_v:v\in V\}$ and $A_E=\{y_e:e\in E\}$. Solutions to $\LL$ in
$A_V\cup A_E$ are in one-to-one correspondence with solutions to $\LL$
in $A'_V\cup A'_E$.

We now construct the set $U$. Start with $U$ being empty, and for each edge of $G$, add $2r$ elements to it.
Suppose that edge $e$ of $G$ joins $v$ and $w$ with $v$ coming before $w$ in the vertex ordering of Lemma~\ref{construct-set}.
For each $i=1,\ldots,r$, add the elements $u_{v,e,i}=x_v + N_{e,i}a_2$ and $u_{w,e,i}=x_w - N_{e,i}a_1$ to $U$. Finally set $A=A_V\cup A_E \cup U$.

Notice that all the conditions on $A$ in the statement of the lemma are satisfied except possibly Condition~5.

Given a $3$-tuple $(z_1,z_2,z_3)$, let
\[ \LL(z_1,z_2,z_3) = a_1 z_1 + a_2 z_2 - b_0 z_3.\]
If $z\in A$, then define $\bar z = z$ if $z\notin U$ and otherwise if $z = x_v \pm N_{e,i} a_k$ then define $\bar z=x_v$.

For any $3$-tuple $(z_1,z_2,z_3)$ of elements of $A$,
\begin{equation} \label{eq:diff} |\LL(z_1,z_2,z_3 ) - \LL(\bar z_1,\bar z_2,\bar z_3)| \leq 3c^2N'<N.\end{equation}

Suppose that $(z_1,z_2,z_3)$ is a solution to $\LL$ in $A$. Then $\LL(z_1,z_2,z_3)=0$. Furthermore, each of $\bar z_1,\bar z_2,\bar z_3$ is divisible by $N$, so $\LL(\bar z_1,\bar z_2,\bar z_3 )$ is divisible by $N$.
Combining this observation with~\eqref{eq:diff}, we deduce that $\LL(\bar z_1,\bar z_2,\bar z_3)=0$. So $(\bar z_1,\bar z_2,\bar z_3)$ is a solution to $\LL$ in $A_V \cup A_E$, and we observed earlier that by dividing each element by $N$ we obtain from it a solution to $\LL$ in $A'_V \cup A'_E$.
Hence either $\{\bar z_1, \bar z_2, \bar z_3\} = \{x_v,x_w,y_e\}$ for some edge $e=vw$ or $\LL$ is translation invariant and $(\bar z_1, \bar z_2, \bar z_3) = (z,z,z)$ for some $z\in A_V\cup A_E$.

In the former case, by permuting the indices if necessary, we may assume that $\bar z_1=x_v$, $\bar z_2=x_w$ and $\bar z_3=y_e$ with $v$ coming before $w$ in the vertex ordering of Lemma~\ref{construct-set}. Then the choice of the integers $N_{e,i}$ ensures that there exists an edge $e$ and integer $i$ such that $z_1\in \{x_v,x_v+a_2N_{e,i}\}$, $z_2\in \{x_w,x_w-a_1N_{e,i}\}$ and $z_3=y_e$. But no permutation of $(x_v,x_w-a_1N_{e,i},y_e)$ or of
$(x_v+a_2N_{e,i},x_w,y_e)$ can be a solution of $\LL$ because for each permutation $\sigma$ of $\{1,2,3\}$, we have either $\LL(\bar z_{\sigma(1)}, \bar z_{\sigma(2)},\bar z_{\sigma(3)})=0$ or $|\LL(\bar z_{\sigma(1)}, \bar z_{\sigma(2)},\bar z_{\sigma(3)}|\geq N$. Thus $(z_1,z_2,z_3)=(x_v,x_w,y_e)$ or $(z_1,z_2,z_3)=(x_v+a_2N_{e,i},x_w-a_1N_{e,i},y_e)$.

In the latter case, either $z=y_e$ for some edge $e$ or $z=x_v$ for some vertex $v$. If $z=y_e$, then $z_1=z_2=z_3=y_e$ and we have a trivial solution. 
For any vertex $v$, edge $e$ and integer $i$, the set $A$ contains at
most one of 
$x_v+a_2N_{v,e,i}$ and $x_v-a_1N_{v,e,i}$. So either
$\{z_1,z_2,z_3\}\subseteq \{x_v,x_v+a_2N_{v,e,i}\}$ or
$\{z_1,z_2,z_3\}\subseteq \{x_v,x_v-a_1N_{v,e,i}\}$ for some edge $e$
and integer $i$.
As $\LL$ is translation invariant we have $c_j + c_k \ne 0$ for all $j$ and $k$. Thus in either case $z_1=z_2=z_3$ and we have a trivial solution. Thus Condition~5 holds.
\end{proof}

For the case when $\ell \ge 4$ we have a different construction.

\begin{lemma} \label{construct-set-amplify-edges}
Consider a linear equation
$c_1 x_1 + \cdots + c_{\ell} x_{\ell} = 0$, where $\ell \ge 4$
and the coefficients $c_i$ are all non-zero, and let
$\LL$ be an equivalent linear equation
$a_1 x_1 + a_2 x_2 + b_1 y_1 + \cdots + b_{\ell-3} y_{\ell-3} = b_0y_0$
in standard form.
Let $G = (V,E)$ be a graph and let $r\in \nats$.
Then we can construct in polynomial time a
set $A \subseteq \mathbb{Z}$ with the following properties:
\begin{enumerate}
\item $A$ is the union of sets
$A_V, A_E$, where
$A_V = \{x_v : v \in V\}$ and
$A_E = \{y_{e,j}^i : e \in E, i = 0, 1, \ldots , \ell-3,
j = 1,\ldots,r\}$; thus for each edge $e$ we have $r$ sets of values;
\item $|A| = |V| + (\ell-2)r|E|$;
\item for every edge $e = (v,w)$ and each $j = 1,\ldots,r$, some permutation of
$(x_v,x_w, y_{e,j}^1 , \ldots , y_{e,j}^{\ell-3} , y_{e,j}^0)$ is a solution to
$\LL$; such a solution we call an edge-solution;
\item if $(z_1, \ldots , z_{\ell})$ is a non-trivial solution to
$\LL$, then $(z_1, \ldots , z_{\ell})$ is a permutation of
$(x_v,x_w, y_{e,j}^1 , \ldots , y_{e,j}^{\ell-3} , y_{e,j}^0)$ for some edge $e =
(v,w)$ and some $j \in \{1,\ldots,r\}$, except that, in that case that
$\LL$ is the equation $-x_1 + x_2 + y_1 = y_0$, the sequence
$(z_1,z_2,z_3,z_4)$ may also be a permutation of
$(y_{e,j}^1, y_{e,j}^0, y_{e,j'}^1, y_{e,j'}^0)$ for any edge $e$ and
$j \ne j'$;
\item $\max(A) = \mathcal{O}(|V|^{2\ell+2}r^{\ell+1})$;
\item $A \subseteq \mathbb{N}$ unless all the coefficients
$c_1 , \ldots , c_{\ell}$ have the same sign.
\end{enumerate}
\end{lemma}

\begin{proof}
The proof very largely follows that of Lemma~\ref{construct-set}. As
before we have a variable $X_v$ corresponding to each vertex value
$x_v$ and a variable $Y_{e,j}^i$ corresponding to each edge value
$y_{e,j}^i$; the set of all of these variables is $B$.

As before, consider any sequence
$\s = (W_1, W_2, Z_1, \ldots, Z_{\ell-3}, Z_0) \in B^{\ell}$.
We need to show that if $\rlc(\s)$ is identically zero, then either
$\s$ is an edge-sequence or $\lc(\s)$ is identically zero.
First suppose that $\ell \ge 5$. The only extra case to consider is
where there are two terms in the sequence which are dependent edge
variables for the same edge~$e$, but from different sets, that is, two
variables $Y_{e,j}^0, Y_{e,j'}^0$, with $j \ne j'$, both of which have
non-zero total coefficient. But then each requires the other terms
$Y_{e,j}^i, Y_{e,j'}^i$ for each $i \ne 0$ to occur as well, giving at
least $2(\ell-2)$ terms in all, which is impossible if $\ell \ge 5$
since then $2(\ell-2) > \ell$.

The case $\ell = 4$ is a little more complicated, because in this case
there could be further non-trivial solutions in one case only.
Consider any sequence
$\s = (W_1, W_2, Z_1, Z_0) \in B^{4}$, and suppose that
$\rlc(\s)$ is identically zero. Then
$\lc(\s) = a_1 W_1 + a_2 W_2 + b_1 Z_1 - b_0 Z_0$.
As above the only way this can happen, apart from an edge solution or
a trivial solution, is if two of the variables are $Y_{e,j}^0,
Y_{e,j'}^0$, with $j \ne j'$, in which case the other two must be
$Y_{e,j}^1, Y_{e,j'}^1$. Thus, for some permutation $(p_1,p_2,p_3,p_4)$
of $(a_1, a_2, b_1, -b_0)$, we have
\begin{equation}
\label{linear-comb}
\lc(\s) = p_1 Y_{e,j}^0 + p_2 Y_{e,j}^1 + p_3 Y_{e,j'}^0 + p_4 Y_{e,j'}^1 .
\end{equation}
If $e = (v,w)$, where $v < w$, then we know
that these variables must satisfy
\begin{align*}
a_1 X_v + a_2 X_w + b_1 Y_{e,j}^1 &= b_0 Y_{e,j}^0
\\
a_1 X_v + a_2 X_w + b_1 Y_{e,j'}^1 &= b_0 Y_{e,j'}^0  .
\end{align*}
Hence substituting for $Y_{e,j}^0$ and $Y_{e,j'}^0$ in~\eqref{linear-comb}, we find that
\begin{align}
\nonumber
\rlc(\s) &=
p_1
\left(\frac{a_1}{b_0}X_v +\frac{a_2}{b_0}X_w + \frac{b_1}{b_0}Y_{e,j}^1 \right)
+ p_2 Y_{e,j}^1 \\
\label{reduced-linear-comb}
&+ p_3
\left(\frac{a_1}{b_0}X_v +\frac{a_2}{b_0}X_w + \frac{b_1}{b_0}Y_{e,j'}^1 \right)
+ p_4 Y_{e,j'}^1 .
\end{align}
Since this is identically zero, we must have $p_1 + p_3 = 0$, and so
since $p_1 \frac{b_1}{b_0} + p_2 = 0$ and
$p_3 \frac{b_1}{b_0} + p_4 = 0$, we also have $p_2 + p_4 = 0$.
Hence the coefficients are $\pm p, \pm q$ for some $p \ge q \ge 0$,
and so the standard form is
\[
-q x_1 + q x_2 + p y_1 = p y_0 .
\]
Hence we have $b_1 = b_0$, so that $p_1 + p_2 = 0$. Hence all the
coefficients are $\pm p_1$, so $p=q$ and so
(after dividing by $p$), the equation in standard form is
\[
- x_1 + x_2 + y_1 = y_0 .
\]
Thus in this case only, we have $\lc(\s)$ is not identically zero, and so
the values $y_{e,j}^0, y_{e,j}^1, y_{e,j'}^0, y_{e,j'}^1$ for any
$j \ne j'$, give further non-trivial solutions to $\LL$ in $A$.
\end{proof}

The problem $\#$\textsc{Independent Set} is defined as follows.

\medskip
\begin{framed}
\noindent $\#$-\textsc{Independent Set}\newline
\textit{Input:} A graph $G$. \newline
\textit{Output:} The number of independent sets of $G$.
\end{framed}

When stated in terms of logical variables and clauses, this problem is better known as $\#$\textsc{Monotone 2-SAT}.
The equivalence of the two is demonstrated by replacing logical variables by vertices and clauses by edges. Then satisfying assignments correspond to independent sets via the correspondence mapping the set of variables assigned false in a truth assignment to the corresponding subset of the vertices. It was shown to be $\#P$-complete by Valiant~\cite{Valiant}.

We shall need the following simple lemma similar to part of Fact~6 from~\cite{Valiant}.
\begin{lemma}\label{lem:easyNT}
Let $z_0,\ldots,z_m$ be non-negative integers and let $p > \sum_{t=0}^m z_t$ with  $p\in \rats$. Then for each $i$, $z_i$ is uniquely determined by $\sum_{t=0}^m z_t p^t$ and $z_{i+1},\ldots,z_m$, and may be found in polynomial time.
\end{lemma}

\begin{proof}
If $\sum_{t=0}^m z_t p^t$ and $z_{i+1},\ldots,z_m$ are known, then one may compute
$\sum_{t=0}^i z_t p^t$. Moreover $p^i > p^{i-1} \sum_{t=0}^m z_t \geq \sum_{t=0}^{i-1} z_t p^t$. Hence $z_i = \lfloor
\sum_{t=0}^{i} z_t p^t / p^i\rfloor$.
\end{proof}

Now we can show $\#\LL$-\textsc{Free Subset} is $\#P$-complete for any
homogeneous equation with at least three variables. Again there are
two cases, using each of the constructions above.

\begin{thm}~\label{thm:counthom}
Let $\LL$ be the linear equation $c_1x_1+c_2x_2+c_3x_3 =0$,
where the coefficients $c_i$ are all non-zero.
Then $\#\LL$-\textsc{Free Subset} is $\#P$-complete. If
the coefficients of $\LL$ are not all of the same sign, then the input
set $A$ can be restricted to be a subset of $\nats$.
\end{thm}

\begin{proof}
We reduce from $\#$\textsc{Independent Set}. Let $G=(V,E)$ be an instance of $\#$\textsc{Independent Set}.
Let $m=|E|$, $n=|V|$ and $r$ be the smallest integer such that
\[ r > (m+n) \frac {\ln 2}{\ln (4/3)}.\]
We construct, in polynomial time, an instance of $\#\LL$-\textsc{Free Subset} by applying Lemma~\ref{lem:counting3.2}. If the coefficients of $\LL$ are not all of the same sign, then the set given by Lemma~\ref{lem:counting3.2} is a subset of $\nats$. Write $A=A_E\cup A_V \cup U$, using the same notation as Lemma~\ref{lem:counting3.2}.
We shall count the number of ways that an $\LL$-free subset $S$ of $A_E\cup A_V$ may be extended to an $\LL$-free subset of $A$. Notice that the only solutions to $\LL$ in $A$ involving either
$u_{v,e,i}$ or $u_{w,e,i}$
include $y_e$ and both $u_{v,e,i}$ and $u_{w,e,i}$. Thus if $y_e\in S$, then one or other
but not both of $u_{v,e,i}$ and $u_{w,e,i}$ may be added to $S$ while maintaining $\LL$-freeness.
If $y_e\notin S$, then any subset of $\{u_{v,e,i},u_{w,e,i}\}$ may be added to $S$ while maintaining $\LL$-freeness. Suppose that $|S\cap A_E|=t$. Then $S$ may be extended to an $\LL$-free subset of $A$ in $3^{rt}4^{r(m-t)}$ ways. Let $z_t$ denote the number of $\LL$-free subsets of $A_E\cup A_V$ with $|S\cap A_E|=m-t$.
Then the number of $\LL$-free subsets of $A$ is
\[ \sum_{t=0}^m 3^{rt}4^{r(m-t)}z_{m-t} = \sum_{t=0}^m 3^{rm} (4/3)^{rt}z_{t} .\]
Now let $S$ be any subset of $A_E\cup A_V$. If $A_E\subseteq S$, then $S$ is $\LL$-free if and only if $S\cap A_V$ is an independent set of $G$.
Thus the number of independent sets of $G$ equals $z_0$. Lemma~\ref{lem:easyNT} implies that
providing $(4/3)^r > \sum_{t=0}^m z_t$, $z_0$ may be found from $\sum_{t=0}^m (4/3)^{rt}z_{t}$ which itself is easily computed from the number of $\LL$-free subsets of $A$.
As $\sum_{t=0}^m z_t \leq  2^{|A_E|+|A_V|} = 2^{m+n}$, it is sufficient to choose
$r$ so that $r > (m+n) \frac {\ln 2}{\ln (4/3)}$.
\end{proof}

\begin{thm}~\label{thm:counthomlge4}
Let $\LL$ be the linear equation $c_1x_1+\ldots+c_\ell x_\ell
=0$, where $\ell \geq 4$ and the coefficients $c_i$ are all non-zero.
Then $\#\LL$-\textsc{Free Subset} is $\#P$-complete. If the
coefficients of $\LL$ are not all of the same sign, then the
input set $A$ can be restricted to be a subset of $\nats$.
\end{thm}

\begin{proof}
We reduce from $\#$\textsc{Independent Set}. Let $G=(V,E)$ be an
instance of $\#$\textsc{Independent Set}.
Let $m=|E|$, $n=|V|$.
We construct, in polynomial time, an instance of $\#\mathcal
L$-\textsc{Free Subset} by applying Lemma~\ref{construct-set-amplify-edges}. If
the coefficients of $\LL$ are not all of the same sign, then
the set given by Lemma~\ref{construct-set-amplify-edges} is a subset of $\nats$.
Write $A=A_V\cup A_E$, using the same notation as
Lemma~\ref{construct-set-amplify-edges}.
We shall count the number of ways that an $\LL$-free subset $S$
of $A_V$ may be extended to an $\LL$-free subset of $A$.
We will say that a subset $S$ of $A_V$ contains an edge $e=(v,w)$ if
$x_v, x_w \in S$.

We first deal with the case where $\LL$ is not equivalent to the
equation $x_1 + x_2 = x_3 + x_4$.
Consider a subset $S \subseteq A_V$ which
contains~$t$ edges. Then for each such edge~$e \in S$, we can add any set
of edge values $y_{e,j}^i$ to $S$ provided that we do not add a full
set $\{y_{e,j}^1, \ldots , y_{e,j}^{\ell-3}, y_{e,j}^0\}$ for any $j$,
hence there is a choice of $(2^{\ell-2}-1)^r$ subsets for each of
these $t$ edges.
On the other hand, for any edge $e$ not in $S$, we can add any subset
of the edge values for $e$, hence there is a choice of
$(2^{\ell-2})^r$ subsets. Thus if $z_t$ is the number of subsets of
$A_V$ containing $m-t$ edges, then the number of $\LL$-free subsets of
$A$ is
\[
\sum_{t=0}^m (2^{\ell-2}-1)^{r(m-t)}(2^{\ell-2})^{rt}z_{t}
=
(2^{\ell-2}-1)^{rm}
\sum_{t=0}^m z_{t}
\left(\frac{2^{\ell-2}}{2^{\ell-2}-1}\right)^{rt} .
\]
Hence if $r$ is the smallest integer such that
\[
r > \frac {n \ln 2}{\ln (2^{\ell-2}) - \ln(2^{\ell-2}-1)},
\]
then by Lemma~\ref{lem:easyNT}, we can determine $z_m$ from this. But
$z_m$ is the number of subsets of $A_V$ containing no edges, that is, the
number of independent sets of $G$, as required.

Now consider the special case of the equation $-x_1 + x_2 + y_1 = y_0$.
Again consider a subset $S \subseteq A_V$ which contains~$t$ edges. As
before, for each of these edges, we can add any set of edge values
$y_{e,j}^i$ to $S$ provided that we do not add a full set
$\{y_{e,j}^1, y_{e,j}^0\}$ for any $j$, hence there is a choice of
$3^r$ subsets for each of these $t$ edges.

For any edge $e$ not in $S$, we can add all of these sets, but we could
also add any set which includes $\{y_{e,j}^1, y_{e,j}^0\}$ for exactly
one value of $j$. Hence in total we can add $3^r + r.3^{r-1}$ subsets.
No other sets are possible because these would contain the special
extra non-trivial solutions to $\LL$. Thus as before, if $z_t$ is the
number of subsets of $A_V$ containing $m-t$ edges, then the number of
$\LL$-free subsets of $A$ is
\[
\sum_{t=0}^m (3^r)^{m-t}((r+3)3^{r-1})^{t}z_{t}
=
3^{mr} \sum_{t=0}^m \left(\frac{r+3}{3}\right)^t z_{t} .
\]
Hence if we can determine this number for $m+1$ different values of
$r$, we have a system of $m+1$ equations in the quantities $z_0,
\ldots, z_m$. Since the coefficients have a non-zero Vandermonde
determinant, the solution is unique and can be determined in
polynomial time in the total size of the coefficients~\cite{Edmonds}.
Hence we can determine $z_m$ as before.
\end{proof}

We now consider the counting version of $\LL$-\textsc{Free Subset}
when $\LL$ is an inhomogeneous equation. We shall need a final variant of
Lemma~\ref{construct-set}.

\begin{lemma}\label{lem:counting3.2inhom}
Consider a linear equation $c_1x_1+\cdots+c_\ell x_\ell =K$, where $\ell \geq 3$ and the coefficients $c_i$ and constant $K$ are all non-zero, with $\gcd(c_1,\ldots,c_\ell )$ a divisor of $K$.
Let $G=(V,E)$ be a bipartite graph and let $r\in \nats$.
Then we can construct in polynomial time, a set $A\subseteq \ints$ with the following properties.
\begin{enumerate}
\item $A$ is the union of $\ell $ sets
$A_V,\,A_E^0,\,A_E^1,\ldots,A_E^{\ell -3}$ and $U$, where
$A_V=\{x_v:v\in V\}$, $A_E^i=\{y_e^i:e\in E\}$ for each
$i=0,\,1,\ldots,\ell -3$ and
$U=\{u_{v,e,i}:v\in V,\, e=vw\in E,\,1\leq i\leq r\}$;
\item $|A| = |V| + (\ell-2+2r)|E|$;
\item for every edge $e=vw$, some permutation of
$(x_v,x_w,y_e^1,\ldots,y_e^{\ell -3},y_e^0)$ is a solution to $\LL$;
\item for every edge $e=vw$ and every integer $i$ with $1\leq i\leq r$,
some permutation of $(u_{v,e,i},u_{w,e,i},y_e^1,\ldots,y_e^{\ell
-3},y_e^0)$ is a solution to $\LL$;
\item if $(z_1,\ldots,z_\ell )$ is a solution to $\LL$, then
$(z_1,\ldots,z_\ell )$ is a permutation of
$(x_v,x_w,y_e^1,\ldots,y_e^{\ell -3},y_e^0)$ for some edge $e=vw$ or
of $(u_{v,e,i},u_{w,e,i},y_e^1,\ldots, y_e^{\ell -3},y_e^0)$ for some
edge $e=vw$ and integer $i$ with $1 \leq i \leq r$.
\item $\max(A)=\mathcal{O}(|V|^{2\ell +2}r   )$;
\item $A\subseteq \nats$ unless all the coefficients
$c_1,\ldots,c_\ell $ have the same sign.
\end{enumerate}
\end{lemma}

\begin{proof}
If the coefficients of $\LL$ do not all have the same sign, then we may assume that $c_1$ is positive and $c_2$ is negative.
Apply Lemma~3.8, by regarding $G$ as a tripartite graph with $V_3=\emptyset$, giving sets $A_V$ and $A_E$ as described in Lemma~3.8.
Let $c=\max\{|c_1|,\ldots,|c_\ell |\}$ and $N=2\ell c\max_{a \in A}|a|+ 1$.

Next for each $e$ in $E$ and integer $i$ with $1\leq i \leq r$ choose a strictly positive integer $N_{e,i}$
so that if $(e_1,i_1) \neq (e_2,i_2)$ and $j_1\ne j_2$, then the following holds
\begin{equation}
c_{j_1}c_1N_{e_1,i_1} - c_{j_2}c_2N_{e_2,i_2} \ne 0.  \label{eq:constinhom}
\end{equation}
In particular, the integers $N_{e,i}$ are pairwise distinct.
We may choose the integers $N_{e,i}$ greedily in any order. When one of them is chosen, there are strictly fewer than $2|E|r \ell^2$ linear equations that it must not satisfy. Hence one may choose
these integers
from $\{1,\ldots, 2|E|r\ell^2\}$.

We now construct the set $U$. Start with $U$ being empty, and for each edge of $G$, add $2r$ elements to it.
Suppose that the edge $e$ of $G$ joins $v$ and $w$, where $v\in V_1$ and $w\in V_2$ in the notation of Lemma~3.8.
For each $i=1,\ldots,r$, add the elements $u_{v,e,i}=x_v - NN_{e,i}c_2$ and $u_{w,e,i}=x_w + NN_{e,i}c_1$ to $U$.

Notice that all the conditions on $A$ in the statement of the lemma are satisfied except possibly Condition~5.

Given an $\ell $-tuple $(z_1,\ldots,z_\ell )$, let $\LL(z_1,\ldots,z_\ell ) = \sum_{i=1}^\ell  c_i z_i$.
If $z\in A$, then define $\bar z = z$ to be the unique element of $A$ to which it is congruent modulo $N$.
Suppose that $(z_1,\ldots,z_\ell )$ is an $\ell $-tuple of elements of $A$ satisfying
$\LL(z_1,\ldots,z_\ell )=K$. Then $\LL(\bar z_1,\ldots,\bar z_\ell )=K \pmod N$. But $|\LL(\bar z_1,\ldots,\bar z_\ell )|<N/2$, so $\LL(\bar z_1,\ldots,\bar z_\ell )=K$ and $(\bar z_1,\ldots,\bar z_\ell )$ is a permutation of an edge solution to $\LL$. Exactly two elements of $\{z_1,\ldots,z_\ell \}$ belong to $U \cup A_V$. Suppose without loss of generality that $z_r = x_v \pmod N$ and $z_s = x_w \pmod N$, with $v\in V_1$ and $w\in V_2$ in the notation of Lemma~3.8.
Then $z_r=x_v$ and $z_s=x_w$, or $z_r=x_v-NN_{e_1,i_1}c_2$ and $z_s=x_w+NN_{e_2,i_2}c_1$. In the latter case, the choice of the integers $N_{e,i}$ ensures that $e_1=e_2=vw$ and $i_1=i_2$.
\end{proof}

\begin{thm} \label{thm:countinhom}
Let $\LL$ be a linear equation $c_1x_1+\ldots+c_\ell x_\ell =K$, where
$\ell \geq 3$ and $K$ and the coefficients $c_i$ are all non-zero. Then
$\#\LL$-\textsc{Free Subset} is $\#P$-complete. If the coefficients of
$\LL$ are not all of the same sign, then the input set $A$ can be
restricted to be a subset of $\nats$.
\end{thm}

\begin{proof}
The proof is very similar to that of Theorem~\ref{thm:counthom},
except that this time we apply Lemma~\ref{lem:counting3.2inhom} and
use the fact that $\#$\textsc{Independent Set} remains
$\#$P-hard when the input is restricted to being a bipartite
graph~\cite{ProvanBall}.
\end{proof}

Combining Theorems~\ref{thm:counthom}, \ref{thm:counthomlge4} and
\ref{thm:countinhom} gives the following:

\begin{thm} \label{thm:countcombined}
Let $\LL$ be a linear equation $c_1x_1+\ldots+c_\ell x_\ell =K$, where
$\ell \geq 3$ and the coefficients $c_i$ are all non-zero. Then
$\#\LL$-\textsc{Free Subset} is $\#P$-complete. If the coefficients of
$\LL$ are not all of the same sign, then the input set $A$ can be
restricted to be a subset of $\nats$.
\end{thm}

%
%
%
%

\end{document}